\documentclass[DIV=14,letterpaper]{scrartcl}

\usepackage{amsmath,amssymb,amsfonts,amsthm,subfig} 
\usepackage{color}                    

\usepackage{graphicx}
\def \D{\mathbb{D}}
\def \DK{\mathbb{D}_K}
\def \M{\mathbb{M}}
\def \MK{\mathbb{M}_K}
\def \Th{\mathcal{T}_h}
\newcommand{\V}[1]{\mbox{\boldmath $ #1 $}}

\newcommand{\bey}{\begin{eqnarray}}
\newcommand{\eey}{\end{eqnarray}}
\newcommand{\nn}{\nonumber}
\newcommand{\beq}{\begin{equation}}
\newcommand{\eeq}{\end{equation}}
\theoremstyle{plain}
\newtheorem{thm}{\hspace{6mm}{\textbf Theorem}}[section]

\newtheorem{lem}{\hspace{6mm}{\textbf Lemma}}[section]

\theoremstyle{definition}

\theoremstyle{remark}
\newtheorem{exam}{\hspace{6mm}{\textbf Example}}[section]
\newtheorem{rem}{\hspace{6mm}{\textbf Remark}}[section]

\vspace{5mm}

\begin{document}

\date{}
\title{A study on nonnegativity preservation in finite element approximation of Nagumo-type nonlinear differential equations}

\author{Xianping Li%
\thanks{Department of Mathematics and Statistics, the University of Missouri-Kansas City, Kansas City, MO 64110, U.S.A. (\textit{lixianp@umkc.edu}) }
\and Weizhang Huang%
\thanks{Department of Mathematics, the University of Kansas, Lawrence, KS 66045, U.S.A. (\textit{whuang@ku.edu}) }
}

\maketitle

\vspace{10pt}

\begin{abstract}
Preservation of nonnengativity and boundedness in the finite element solution of Nagumo-type equations with general anisotropic diffusion is studied. 
Linear finite elements and the backward Euler scheme are used for the spatial and temporal discretization, respectively. An explicit, an implicit, and two hybrid explicit-implicit treatments for the nonlinear reaction term are considered. Conditions for the mesh and the time step size are developed for the numerical solution to preserve nonnegativity and boundedness. The effects of lumping of the mass matrix and the reaction term
are also discussed. The analysis shows that the nonlinear reaction term has significant effects on the conditions for both the mesh and the time step size. Numerical examples are given to demonstrate the theoretical findings.
\end{abstract}



\vspace{10pt}

\section{Introduction}
\label{sec-intro}

We consider Nagumo-type equations in the form
\bey
\label{nagumo}
&& \frac{\partial u}{\partial t}  -  \nabla \cdot (\D \nabla u) = u f(u),  \quad \text{ in }\quad \Omega_T = \Omega \times (0,T] \\
\label{bc}
&& u(\V{x},t)   =  g(\V{x},t), \qquad \qquad \quad \text{ on }\quad  \partial \Omega \times (0, T] \\
\label{ic}
&& u(\V{x},0)  =  u_0(\V{x}),  \qquad  \qquad \quad \text{\,  in }\quad \Omega
\eey
where $\Omega \subset \mathbb{R}^d$ is a connected, bounded polygonal or polyhedral domain ($d=1, 2$ or $3$ is the space dimension),
$\V{x}=(x_1, \cdots,x_d)^T$ denotes the coordinates on $\Omega$, $T > 0$ is a fixed time, $f(u)$, $g(\V{x},t)$, and $u_0(\V{x})$
are given functions, and $\D$ denotes the diffusion tensor that can be isotropic (when it is in the form of $\D = \alpha I$, with
$\alpha$ being a scalar function and $I$ being the identity matrix) or anisotropic.
We assume that $\D$ is  continuous and symmetric and uniformly positive definite on $\Omega$.
In our analysis, $f(u)$  and $f'(u)$ are assumed to be continuous and bounded for all bounded $u$. 
A special example is the well-known Nagumo equation \cite{Mck70,NAY62}
where $f(u)=(1-u)(u-a)$ and $a \in (0,1)$ is a parameter.  It can be shown (using an argument similar to the proof
of the maiximum principle for linear parabolic equations \cite{Evans-1998}) that the solution of the problem \eqref{nagumo}, \eqref{bc}, and \eqref{ic} is nonnegative when $g$ and $u_0$ are nonnegative.
Moreover, for the case of the Nagumo equation where $f(1) = 0$,
we have the boundedness $u \le 1$ when $g \le 1$ and $u_0 \le 1$.

Nagumo-type equations arise in various fields including physiology, biological and chemical processes,
and ecology; e.g., see \cite{AA06, GK09, KS98, LEM10, MK14, RP12, TVBN15, WGK11, Yan85}.
The Nagumo equation also appears as one of the two coupled equations in the well-known
Fitzhugh-Nagumo model \cite{Dik05, Fit61, GK10, KRS04, LV06}. 
Understanding the former will help with understanding the latter.
The numerical solution of Nagumo-type equations and corresponding models has been studied extensively
in the past; e.g., see
\cite{Abb08, CGM03, Den91,  Dik05, Has76, Jon84, KSS97, LV06, Mck70, QDD15, RM11}.
However, very little attention has so far been paid to the studies of preservation of solution nonnegativity and boundedness
in the discretization of Nagumo-type equations. The existing work includes \cite{CGM03, Mac12, QDD15, RM11} where
only finite difference discretization and isotropic diffusion have been considered.
On the other hand, preservation of solution nonnegativity and boundedness has important physical implications
and has attracted considerable attention from researchers in the recent years
\cite{CGM03, FH06, FHK11, FKK12, LH10, LH13, LHQ14, LePot09,  WaZh11, YuSh2008, ZS11}. 
For example, Li and Huang \cite{LH10, LH13} have developed conditions for general linear diffusion equations
for the linear finite element solution to satisfy maximum principle (MP). 
The schemes in \cite{LePot09,YuSh2008} (finite volume methods) and
\cite{ZS11} (discontinuous Galerkin schemes)
can be used for nonlinear equations but only allow a small time step due to their explicit time integration.

The objective of this work is to study preservation of solution nonnengativity and boundedness in the finite element
approximation of Nagumo-type equations with general anisotropic diffusion.
We consider anisotropic diffusion here because Nagumo-type equations with anisotropic diffusion can arise from various applications. For example, in cardiac electrophysiology,
the conductivity tensor varies with location and direction. The typical conductivities in cardiac tissue are 0.05 m/s for the sinoatrial and the atrioventricular node, 1 m/s for the atrial pathways, the His bundle and the ventricular muscle bundle, and 4 m/s in the Purkinje fibers \cite{GK09}. In chemical systems with excitable and oscillatory media, the diffusion tensor is also heterogeneous and anisotropic \cite{MS06}.
We shall use linear finite elements and the backward Euler scheme
for the spatial and temporal discretization, respectively. Four treatments of the nonlinear reaction
term will be considered, including an explicit, an implicit, and two hybrid explicit-implicit treatments.
Conditions for the mesh and the time step size will be developed for the numerical solution to preserve
nonnegativity and boundedness. The effects of lumping of the mass matrix and the reaction term
will also be discussed. It is emphasized that the current work is a nontrivial extension of our previous work
\cite{LH13} where only linear diffusion equations have been studied.
As will be seen, the nonlinear reaction term can have significant effects on the mesh and time step size conditions.

The rest of the paper is organized as follows. Sect.~\ref{sec-fem} introduces the finite element formulation
for (\ref{nagumo}). Conditions for nonnegativity preservation using different treatments of Nagumo nonlinearity
are developed in Sect.~\ref{sec-cond}. Lumping for the mass matrix and the reaction term is discussed
in Sect.~\ref{sec-lumping}, followed by the investigation of  boundedness preservation
in Sect.~\ref{sec-boundedness}. Numerical results are presented in Sect.~\ref{sec-ex},
and conclusions are drawn in Sect.~\ref{sec-con}.

\section{Finite element formulation}

\label{sec-fem}

In this section we describe the linear finite element approximation of the Nagumo-type equation (\ref{nagumo}). 
Assume that an affine family of simplicial triangulations $\{ \Th \}$ is given for $\Omega$. Define
\[
U_g = \{ v \in H^1(\Omega),\quad  v|_{\partial \Omega} = g\}.
\]
Let $g_h$ be a piecewise linear approximation of $g$ on $\Th$.
We denote the linear finite element space associated with $\Th$ and $g_h$ by $U^h_{g_h}$.
A linear finite element solution $u_h(t) \in U^h_{g_h}, \; t \in (0, T]$ for \eqref{nagumo} is defined by
\beq
\label{fem-form}
\int_{\Omega} \frac{\partial u_h}{\partial t} \, v_h  d\V{x} + 
\int_{\Omega} (\nabla v_h)^T \; \D \nabla u_h d\V{x} =
 \int_{\Omega} u_h f(u_h) \, v_h d\V{x}, \quad \forall v_h \in U_0^h
\eeq
where $U_0^h$ is the subspace of the linear finite element space with vanishing boundary values.

The above equation can be cast in matrix form.
Denote the numbers of the elements, vertices, and interior vertices of $\Th$
by $N_e$, $N_v$, and $N_{vi}$, respectively.
For notational simplicity, we assume that the vertices have been ordered in such a way that
the first $N_{vi}$ vertices are the interior vertices.
Let $\phi_j$ be the linear basis function associated with the $j$-th vertex, $\V{x}_j$.
Then we can express $u_h$ as
\beq
\label{soln-approx}
u_h = \sum_{j=1}^{N_{v}} u_j \phi_j .
\eeq
Inserting this into (\ref{fem-form}) and taking $v_h = \phi_i$ ($i=1, ..., N_{vi}$) successively,
we obtain the matrix form of the semi-discrete system as
\beq
\label{fem-sys}
M \, \frac{d\V{u}}{d t} + A \, \V{u} = \V{b}(u_h) + \V{g},
\eeq
where $\V{u} = (u_1,..., u_{N_{vi}}, u_{N_{vi}+1},..., u_{N_v})^T$ is the unknown vector and $M$ and $A$
are the mass and stiffness matrices, respectively. The entries of the matrices are given by
\begin{align}
\label{matM}
& m_{ij} = \begin{cases}
	\int_{\Omega} \phi_j \phi_i \, d\V{x} = \sum\limits_{K \in \Th} \int_K \phi_j \phi_i \, d\V{x}, & i=1, ..., N_{vi} \\
	0, & i=N_{vi}+1, ..., N_{v} 
\end{cases} 
\\
\label{matA}
& a_{ij} = \begin{cases}
	\int_\Omega (\nabla \phi_i)^T \; \D \nabla \phi_j \, d\V{x}
	= \sum\limits_{K \in \Th} |K| (\nabla \phi_i )^T \; \DK  \nabla \phi_j , & i=1, ..., N_{vi} \\
	\delta_{ij}, & i=N_{vi}+1, ..., N_{v}
\end{cases}
\end{align}
where $j=1, ..., N_{v}$, $\DK$ is the average of $\D$ over $K$, and $|K|$ denotes the volume of $K$.
The right-hand side vectors $\V{b}(u_h)=(b_i)$ and $\V{g}=(g_i)$ are given by
\begin{align}
\label{vecb}
& b_i = \begin{cases}
	\int_\Omega u_h f(u_h) \phi_i \, d\V{x} = \sum\limits_{K \in \Th} \int_K u_h f(u_h) \phi_i \, d\V{x} , & i=1, ..., N_{vi} \\
	0, & i=N_{vi}+1, ..., N_{v} 
\end{cases}  
\\
\label{vecg}
& g_i = \begin{cases}
	0, & i=1, ..., N_{vi} \\
	g(\V{x}_i,t), & i=N_{vi}+1, ..., N_{v} .
\end{cases}
\end{align}

For the time discretization we denote the numerical approximation of the solution at $t = t_n$ by $u_h^n$.
Applying the backward Euler method to all but the reaction term in \eqref{fem-sys}, we get
\beq
\label{fem-dt}
M \, \frac{\V{u}^{n+1}-\V{u}^n}{\Delta t_n} + A \, \V{u}^{n+1} = \V{\tilde{b}}(\V{u}^n,\V{u}^{n+1}) + \V{g}^{n+1},
\eeq
where $\Delta t_n = t_{n+1}-t_n$, $\V{g}^{n+1} = \V{g}(t_{n+1})$,
and $\V{\tilde{b}}(\V{u}^n,\V{u}^{n+1})$ is an approximation of $\V{b}(u_h)$ for the time step. 

We investigate four approximations in the current work.
The first one (called the explicit method or EM) is to define  $\V{\tilde{b}}(\V{u}^n,\V{u}^{n+1}) = \V{b}(u_h^n)$.
This explicit treatment has been commonly used in the so-called implicit-explicit integration of semi-linear parabolic
equations; e.g., see \cite{ARS97, ARW95, CS10, KCGH07}. 
The second approximation is a fully implicit method (IM) that treats the product implicitly, i.e.,
$\V{\tilde{b}}(\V{u}^n,\V{u}^{n+1}) = \V{b}(u_h^{n+1})$. 
The third and last approximations are hybrid explicit-implicit methods (HEIM) that
use an explicit treatment for some factors and implicit treatment for others in the product. 
The detail of these approximations and their effects
on the preservation of nonnegativity and boundedness of the solution will be discussed in the next section.

\section{Preservation of nonnegativity}
\label{sec-cond}

In this section we describe four approximations of the reaction term and
study the conditions on the mesh and the time step under which
the numerical solution of the system (\ref{fem-sys}) preserves the nonnegativity of the solution
of the continuous problem. 
We assume that $f(u)$ and $f'(u)$ exist and are continuous  and bounded for any  bounded $u \in \mathbb{R}$.

\subsection{Dihedral angles and nonobtuse angle conditions}
\label{sec-notation}

Consider a generic element $K \in \Th$ and denote its vertices by $\V{x}_{0}, \, ...,\, \V{x}_d$.
Note that it is more proper to denote these vertices by $\V{x}_{0}^K,\, ...,\, \V{x}_d^K$.
For notational simplicity we suppress the superscript $K$ as long as no confusion is caused.
This convention will apply to other related quantities.
Define the edge matrix (of size $d \times d$) of $K$  as 
\beq
\label{matE}
E = [\V{x}_1-\V{x}_0, \cdots, \V{x}_d-\V{x}_0].
\eeq
Notice that $E$ is non-singular as long as $K$ is not degenerate. Using the edge matrix we can define
the so-called $\V{q}$-vectors as
\beq
\label{q-vector}
[\V{q}_1, \cdots, \V{q}_d] = E^{-T}, \quad \V{q}_0 = - \sum_{i=1}^d \V{q}_i .
\eeq
Denote the face opposite to vertex $\V{x}_i$ (i.e., the face not having $\V{x}_i$ as its vertex) by $S_i$.

\begin{lem}
\label{lem-qnormal}
The vector $\V{q}_i$ is normal to the face $S_i$ for $i=0,1,\cdots, d$. 
\end{lem}

\begin{proof}
For $i=1,\cdots,d$, from the definition \eqref{q-vector} we have
\[
\V{q}_i^T (\V{x}_j-\V{x}_0) = \delta_{ij}, \quad j = 1, ..., d
\]
which implies that $\V{q}_i$ is orthogonal to $(d-1)$ edges of $S_i$ and therefore orthogonal to $S_i$ itself.
For $\V{q}_0$, we have, for $i, j= 1, ..., d$,
\bey
\nn
\V{q}_0^T (\V{x}_i - \V{x}_j) &=& \V{q}_0^T (\V{x}_i - \V{x}_0) - \V{q}_0^T (\V{x}_j - \V{x}_0) \\
\nn
&=& - \sum_{k=1}^d \V{q}_k^T (\V{x}_i - \V{x}_0) + \sum_{k=1}^d \V{q}_k^T (\V{x}_j - \V{x}_0) \\
\nn
&=& - \sum_{k=1}^d \delta_{ki} + \sum_{k=1}^d \delta_{kj}
= -1 + 1 = 0,
\eey
which implies that $\V{q}_0$ is orthogonal to $S_0$. 
\end{proof}

\begin{lem}
\label{lem-qgradient}
The vector $\V{q}_i$ is equal to 
the gradient of the linear basis function at $\V{x}_i$, i.e., $\V{q}_i = \nabla \phi_i $,
for $i=0,1,\cdots, d$.
\end{lem}

\begin{proof}
We first prove for the result for $i=1,\cdots,d$. From the definition of the linear basis functions, we have
\[
\sum_{i=0}^d \phi_i = 1, \qquad \sum_{i=0}^d \phi_i \V{x}_i = \V{x}. 
\]
Combining them we get
\[
\sum_{i=1}^d (\V{x}_i -\V{x}_0) \phi_i = \V{x} - \V{x}_0.
\]
Differentiating both sides with respect to $\V{x}$, we have
\[
\sum_{i=1}^d (\V{x}_i -\V{x}_0) (\nabla \phi_i)^T = I,
\]
where $I$ is the $d \times d$ identity matrix. This equation can be rewritten in matrix form as
\[
E [\nabla \phi_1, \cdots, \nabla \phi_d]^T = I ,
\]
which implies that
\[ 
[\nabla \phi_1, \cdots, \nabla \phi_d] = E^{-T} .
\]
From the definition of the $\V{q}$-vectors (\ref{q-vector}), we know  $\V{q}_i = \nabla \phi_i$ for $i=1,..., d$.
Moreover, differentiating $\sum\limits_{i=0}^d \phi_i = 1$, we have
\[
\nabla \phi_0 = - \sum_{i=1}^d \nabla \phi_i = - \sum_{i=1}^d \V{q}_i = \V{q}_0. 
\]
\end{proof}

\begin{figure}[thb]
\centering
\includegraphics[width=2.5in]{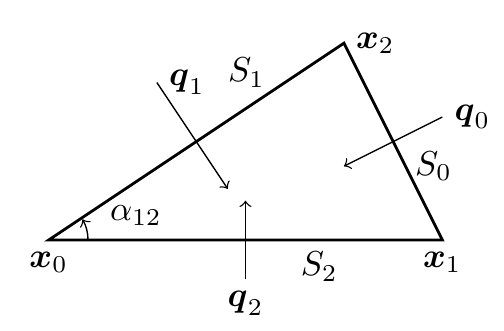}
\caption{The $\V{q}$-vectors, faces, and the dihedral angle $\alpha_{12}$
between $S_1$ and $S_2$ for a triangular element.}
\label{fig:triangle}
\end{figure}

From the above two lemmas one can see that $\V{q}_i$ ($i= 0, 1, ..., d$) are along the fastest ascent direction
(i.e., the direction pointing to the vertex $\V{x}_i$) and thus the inward normal direction of $S_i$.
The $\V{q}$-vectors and the faces for a triangular element are illustrated in Fig.~\ref{fig:triangle}.

Recall that the dihedral angles are defined as the angles between
any two different faces. Then, from the above lemmas we can compute the angles using the $\V{q}$-vectors as
\beq
\cos ({\alpha}_{ij}) = - \frac{\V{q}_i^T \V{q}_j} {\| \V{q}_i \| \cdot \|\V{q}_j \|}
= - \frac{(\nabla \phi_i)^T \nabla \phi_j} {\| \nabla \phi_i \| \cdot \|\nabla \phi_j \|}, \quad i \ne j, \quad i,j=0,1,\cdots, d
\label{dihedral}
\eeq
where $\| \cdot \|$ denotes the $l_2$ vector norm.
From this, the well-known nonobtuse angle condition \cite{BKK07,CR73} can be rewritten as
\beq
\label{nonobtuse}
(\nabla \phi_i)^T \nabla \phi_j \le 0, \quad \forall i \ne j, \quad i,j=0,1,\cdots, d, \quad \forall K \in \Th.
\eeq

The expressions of the dihedral angles can also be derived in a similar way for the metric specified
by $\M_K$, the average of a metric tensor $\M$ over $K$. Notice that $\M_K$ is constant on $K$ and a metric
tensor is always assumed to be symmetric and uniformly positive definite on $\Omega$.
Recall that the Riemannian distance in $\M_K$ is defined as
\[
\| \V{x} \|_{\M_K} = \sqrt{\V{x}^T \M_K \V{x}} = \sqrt{(\M_K^{\frac 1 2}\V{x})^T (\M_K^{\frac 1 2} \V{x})}.
\]
Thus, computing the dihedral angles of $K$ in the metric $\M_K$ is equivalent to computing
those for the simplex (denoted by $\tilde{K}$) with vertices $\M_K^{\frac 1 2} \V{x}_0,\; ...,\; \M_K^{\frac 1 2} \V{x}_d$.
The edge matrix and $\V{q}$-vectors of $\tilde{K}$ are given by
\begin{align*}
& \tilde{E} = [\M_K^{\frac{1}{2}}(\V{x}_1-\V{x}_0), \cdots, \M_K^{\frac{1}{2}} (\V{x}_d-\V{x}_0)] = \M_K^{\frac{1}{2}} E,
\\
& [\tilde{\V{q}}_1,\; ...,\; \tilde{\V{q}}_d] = \tilde{E}^{-T} = \M_K^{-\frac{1}{2}} E^{-T}
= [\M_K^{-\frac{1}{2}}\V{q}_1,\; \cdots,\; \M_K^{-\frac{1}{2}}\V{q}_d],
\\
& \tilde{\V{q}}_0 = - \sum_{i=1}^d \tilde{\V{q}}_i = \M_K^{-\frac{1}{2}}\V{q}_0 .
\end{align*}
Using these, the dihedral angles  of $K$ in the metric $\M_K$ can be expressed as
\beq
\label{dihedral-M-0}
\cos (\tilde{\alpha}_{ij}) = - \frac{\V{\tilde{q}}_i^T \V{\tilde{q}}_j} {\| \V{\tilde{q}}_i \| \cdot \|\V{\tilde{q}}_j \|}
= - \frac{\V{q}_i^T \M_K^{-1} \V{q}_j} {\| \V{q}_i \|_{\M_K^{-1}} \cdot \|\V{q}_j \|_{\M_K^{-1}}}
= - \frac{(\nabla \phi_i)^T \M_K^{-1} \nabla \phi_j} {\| \nabla \phi_i \|_{\M_K^{-1}} \cdot \|\nabla \phi_j \|_{\M_K^{-1}}}, \quad i \ne j .
\eeq

In our analysis, the metric tensor is chosen as the inverse of the diffusion matrix, viz., $\M = \D^{-1}$. In this case, we have
\beq
\label{dihedral-M}
\cos (\tilde{\alpha}_{ij})
= - \frac{(\nabla \phi_i)^T \DK \nabla \phi_j} {\| \nabla \phi_i \|_{\DK} \cdot \|\nabla \phi_j \|_{\DK}}, \quad i \ne j .
\eeq
Then, the mesh is nonobtuse in the metric $\D^{-1}$ if
\beq
\label{anoac-1}  
(\nabla \phi_i)^T \DK \nabla \phi_j \le 0, \quad \forall i \ne j, \quad  i,j=0,1,\cdots, d, \quad \forall K \in \Th .
\eeq
This is referred to as the anisotropic nonobtuse angle condition (ANOAC). It was first used in \cite{LH10}
for the analysis of preservation of the maximum principle in the linear finite element solution
of linear anisotropic diffusion problems. The following lemma is quoted from \cite{LH10}.

\begin{lem}
\label{lem-anoac}
The stiffness matrix $A$ in \eqref{matA} is an $M$-matrix and has nonnegative row sums
if the mesh satisfies ANOAC (\ref{anoac-1}).
\end{lem}

A nonsingular matrix $A=(a_{ij})$ is said to be an $M$-matrix if (a) $a_{ii} > 0$ and $a_{ij} \le 0$ for $i\neq j$ and
(b) the entries of its inverse are nonnegative.
In our analysis we will use a sufficient condition for $M$-matrices (e.g., see Plemmons \cite{Ple77})
that requires $A$ to be a $Z$-matrix (i.e., $a_{ij} \le 0$) and be strictly diagonally dominant.

To conclude this subsection, we list a few facts that will be needed in our analysis. We have
(e.g., see Ciarlet \cite[Page 201]{Cia78})
\beq
\label{int-phi}
\int_{K} \phi_i \phi_j \, d\V{x} = \frac{|K|}{(d+1)(d+2)},
\qquad \int_{K} \phi_i^2 \, d\V{x} = \frac{2|K|}{(d+1)(d+2)}, \quad i \neq j.
\eeq
Moreover, it can be shown that
\begin{align}
& |K| = \frac{|\det(E)|}{d!}, \quad h_i = \frac{1}{\|\V{q}_i\|} = \frac{1}{\|\nabla \phi_i\|},
\label{h-1}
\\
& |\tilde{K}| = |K| \det(\DK)^{- \frac 1 2}, \quad \tilde{h}_i = \frac{1}{\|\tilde{\V{q}}_i\|} = \frac{1}{\|\nabla \phi_i\|_{\DK}},
\label{h-2}
\end{align}
where $h_i$ and $\tilde{h}_i$ are the distance/height from vertex $\V{x}_i$ to face $S_i$ in the Euclidean metric and
the metric $\DK^{-1}$, respectively.
Furthermore, we define
\beq
\label{Dacute}
\D_{acute} = \left (\frac{|\Omega|}{N_e}\right )^{\frac{2}{d}} \min\limits_{K \in \Th} \min\limits_{\substack{i, j = 0, ..., d\\ i \neq j}} (-(\nabla \phi_i)^T \DK \nabla \phi_j) .
\eeq
From (\ref{dihedral-M}) and (\ref{h-2}), we can rewrite this definition as 
\beq
\label{Dacute-1}
\D_{acute} = \left (\frac{|\Omega|}{N_e}\right )^{\frac{2}{d}}
\min\limits_{K \in \Th} \min\limits_{\substack{i, j = 0, ..., d\\ i \neq j}} 
\frac{\cos (\tilde{\alpha}_{ij})}{\tilde{h}_i \tilde{h}_j} .
\eeq
Here, the factor, $(|\Omega|/N_e)^{\frac{2}{d}}$, which is the squared average element size,
has been added to make the quantity dimensionless.
Using this definition, we can state ANOAC (\ref{anoac-1}) as $\D_{acute} \ge 0$
and the anisotropic {\em acute} angle condition (AAAC) as $\D_{acute} > 0$. 

Consider a family of {\em uniformly acute} meshes $\{\Th\}$ that satisfies
\begin{equation}
0 < \tilde{\alpha}_{ij} \le \tilde\alpha <\frac{\pi}{2}, \quad i, j = 0, ..., d,\quad  i \neq j, \quad \forall K \in \Th,
\quad \forall \Th \in \{ \Th \}
\label{uniformly-acute}
\end{equation}
for some constant $ \tilde\alpha$. For these meshes, we have
\beq
\label{Dacute-2}
\D_{acute} \ge \left (\frac{|\Omega|}{N_e}\right )^{\frac{2}{d}} \frac{\cos( \tilde\alpha)}{\tilde{h}^2} > 0,
\eeq
where $\tilde{h}$ is the maximal element diameter of $\Th$ in the metric $\D^{-1}$
and we have used the fact that $\tilde{h}_i \le \tilde{h}$. 
It is worth pointing out that, if the eigenvalues of $\D$ are bounded uniformly on $\Omega$
from below (away from zero) and above, we have $\tilde{h} = \mathcal{O}(h)$.

\subsection{Explicit method (EM)}
\label{sec-EM}

In this case, the reaction term is calculated explicitly, viz.,
\[
\left. u_h f(u_h)\right |_{t_{n+1}} \approx u_h^{n} f(u_h^{n}).
\]
This gives rise to
\beq
\label{rhs-em}
\tilde{\V{b}}(\V{u}^n,\V{u}^{n+1}) = \V{b}(\V{u}^n).
\eeq
It can be rewritten as
\beq
\label{rhs-emb}
\tilde{\V{b}}(\V{u}^n,\V{u}^{n+1}) =  C(\V{u}^{n}) \, \V{u}^{n},
\eeq
where $C(\V{u}^{n})=(c_{ij})$ is given by
\beq
\label{emmatC}
c_{ij} = \begin{cases}
	\int_\Omega f(u_h^n) \phi_j \phi_i \, d\V{x}, & i=1, ..., N_{vi} \\
	0, & i=N_{vi}+1, ..., N_{v} 
\end{cases} , \quad j=1, ..., N_{v}.
\eeq
Substituting this into \eqref{fem-dt}, we have
\beq
(M + \Delta t_n A) \V{u}^{n+1} = (M + \Delta t_n \, C) \V{u}^{n} + \Delta t_n \, \V{g}^{n+1}.
\label{fem-em}
\eeq

For any function $v = v(\V{x})$, we introduce the positive-negative part decomposition
$v(\V{x}) = v^{+}(\V{x}) + v^{-}(\V{x})$, where $v^{+}(\V{x}) = \max\{ v(\V{x}), 0\}$ and
$v^{-}(\V{x}) = \min\{ v(\V{x}), 0\}$.

\begin{thm}
\label{thm-em}
Assume $u_0(\V{x}) \ge 0$ and $g(\V{x},t) \ge 0$. The scheme \eqref{fem-em} preserves the nonnegativity of the solution
of the continuous problem if the mesh satisfies ANOAC (\ref{anoac-1}), i.e., $\D_{acute} \ge 0$ and the time step satisfies
\beq
\label{condt-em}
\frac{\left (\frac{|\Omega|}{N_e}\right )^{\frac{2}{d}}}{(d+1)(d+2)\D_{acute}} \le \Delta t_n
\le \frac{1}{\max\limits_{\V{x} }|f^{-}(u_h^n)|}, \quad n = 0, 1, ...
\eeq
\end{thm}

\begin{proof}
We shall show that $M+\Delta t_n C$ is a non-negative matrix and $M+\Delta t_n A$ an $M$-matrix. 
The former is obvious for $i=N_{vi}+1, ..., N_{v}$ from the definitions of $M$ \eqref{matM} and $C$ \eqref{emmatC}.
So we only need to consider the situation with $i=1, ..., N_{vi}$ and $j=1, ..., N_{v}$.
Assume that $\V{u}^{n} \ge 0$ (in the component-wise sense).
From \eqref{matM} and \eqref{emmatC}, we have
\beq
\label{rhs-em2}
m_{ij}+\Delta t_n c_{ij} = \sum\limits_{K \in \Th} \int_K (1+\Delta t_n f(u_h^n)) \phi_j \phi_i \, d\V{x}.
\eeq
The right-hand side is guaranteed to be nonnegative if $1+\Delta t_n f(u_h^n) \ge 0$, which holds
when the right inequality of (\ref{condt-em}) is satisfied.

To show that $M+\Delta t_n A$ is an $M$-matrix, we show that it is a $Z$-matrix and strictly diagonally dominant.
For diagonal entries with $i = 1, ..., N_{vi}$, from \eqref{matM}, \eqref{matA} and \eqref{int-phi} we have
\bey
\nn
m_{ii}+\Delta t_n a_{ii} & = & \sum\limits_{K \in \Th} \int_K \phi_i^2 \, d\V{x}
+ \Delta t_n \sum\limits_{K \in \Th} |K| (\nabla \phi_i)^T \; \DK \; \nabla \phi_i \\
\nn
& = & \sum\limits_{K \in \omega_i} \int_K \phi_i^2 \, d\V{x}
+ \Delta t_n \sum\limits_{K \in \omega_i} |K| (\nabla \phi_i)^T \; \DK \; \nabla \phi_i \\
\nn
& \ge & \frac{2 |\omega_i|}{(d+1)(d+2)} > 0,
\eey
where $\omega_i$ is the patch of the elements containing $\V{x}_i$ as a vertex and
we have used the fact that $\DK$ is positive definite.
Similarly, for off-diagonal entries with $i = 1, ..., N_{vi}$, $j=1, ..., N_v$, $i \ne j$, we have
\bey
\label{EM-offdiag}
m_{ij} + \Delta t_n\, a_{ij} & = & 
\sum_{K \in \Th} \int_K \phi_j \phi_i \, d\V{x}
+ \Delta t_n \, \sum\limits_{K \in \Th} |K| (\nabla \phi_i)^T \; \DK \; \nabla \phi_j  \\
\nn
& = & 
\sum_{K \in \omega_i \cap \omega_j} \int_K \phi_j \phi_i \, d\V{x}
+ \Delta t_n \, \sum\limits_{K \in \omega_i \cap \omega_j} |K| (\nabla \phi_i)^T \; \DK \; \nabla \phi_j  \\
\nn
& = & \sum_{K \in \omega_i\cap \omega_j} |K| \left ( \frac{1}{(d+1)(d+2)} 
+ \Delta t_n \;(\nabla \phi_{i})^T \; \DK \; \nabla \phi_{j}  \right )
\\
\nn 
& \le & \sum_{K \in \omega_i\cap \omega_j} |K| \left ( \frac{1}{(d+1)(d+2)} 
- \Delta t_n  \left (\frac{|\Omega|}{N_e}\right )^{-\frac{2}{d}} \D_{acute} \right ) \le 0,
\eey
where we have used the definition of $\D_{acute}$ and the condition (\ref{condt-em}).
It is easy to check that $m_{ii} + \Delta t_n\, a_{ii} > 0$ and $m_{ij} + \Delta t_n\, a_{ij} \le 0$ ($i\neq j$)
for $i = N_{vi}+1, ..., N_v$. Thus, $M + \Delta t_n A$ is a $Z$-matrix.

The diagonal dominance follows from the fact that $M + \Delta t_n A$ is a $Z$-matrix
and that, for $i = 1,..., N_{vi}$,
\bey
\nn
\sum\limits_{j=1}^{N_{v}}(m_{ij} + \Delta t_n \, a_{ij}) & = & 
\int_{\Omega} \phi_i \sum_{j=1}^{N_v} \phi_j d \V{x}
+ \Delta t_n\, \int_{\Omega}(\nabla  \phi_i)^T \D  \sum_{j=1}^{N_v} \nabla \phi_j d \V{x}
\\
\nn
& = &  \int_{\Omega} \phi_i d \V{x} = \frac{|\omega_i|}{(d+1)(d+2)}  > 0 .
\eey
Thus, $M + \Delta t_n A$ is an $M$-matrix.

Using the assumptions $\V{g}^{n+1} \ge 0$ and $\V{u}^n \ge 0$ and the fact that
$(M + \Delta t_n C) \ge 0$ and $M + \Delta t_n A$ is an $M$-matrix, from (\ref{fem-em})
we know that $\V{u}^{n+1} \ge 0$. From the induction, the numerical solution stays nonnegative
for all time.
\end{proof}

\begin{rem}
\label{rem:3.1}
The upper bound of (\ref{condt-em}) does not impose a serious restriction on $\Delta t_n$. For example,
for the case with Nagumo's equation, we have $f(u) = (1-u) (u-a)$. Assuming that the numerical solution stays
in [0,1], we have $\max_{\V{x}} |f^{-}(u_h^n)| \le a$. Thus, the right condition is satisfied if $\Delta t_n \le 1/a$.
Moreover, the lower bound of  (\ref{condt-em}) becomes unbounded if $\D_{acute} = 0$,
which happens if there is a right dihedral angle. Thus, (\ref{condt-em}) effectively
requires $\D_{acute} > 0$, i.e., {\em the mesh be acute}. Furthermore, (\ref{condt-em}) implies that
\begin{equation}
(d+1)(d+2)\D_{acute} \ge \left (\frac{|\Omega|}{N_e}\right )^{\frac{2}{d}} \max\limits_{\V{x}}|f^{-}(u_h^n)| .
\label{condt-em2}
\end{equation}
If the mesh is uniformly acute (cf. (\ref{uniformly-acute})),
from (\ref{Dacute-2}) we know that (\ref{condt-em2}) and (\ref{condt-em}) essentially are
\[
\mathcal{O}(1) \ge \mathcal{O}(h^2),
\qquad \mathcal{O}(h^{2}) \le \Delta t_n \le \mathcal{O}(1),
\]
which can readily be satisfied when the mesh is sufficiently fine.
\qed
\end{rem}

\subsection{Implicit method (IM)}
\label{sec-IM}

Another straightforward treatment for the reaction term is to evaluate it fully implicitly, i.e., 
\[
\left. u_h f(u_h)\right |_{t_{n+1}} \approx u_h^{n+1} f(u_h^{n+1}).
\]
But this will require the solution of nonlinear algebraic systems. To avoid this, we can use
a linearization, for instance,
\[
u_h^{n+1} f(u_h^{n+1}) \approx u_h^{n+1} (f(u_h^n) + u_h^n f'(u_h^n)) - u_h^n f'(u_h^n) u_h^n.  
\]
This gives rise to 
\beq
\label{rhs-imb}
\tilde{\V{b}}(\V{u}^n, \V{u}^{n+1}) =  B \, \V{u}^{n+1} + C \, \V{u}^n,
\eeq
where $B=(b_{ij})$ and $C=(c_{ij})$ are given by 
\begin{align}
\label{immatB}
& b_{ij} = \begin{cases}
	\int_\Omega (f(u_h^n)+u_h^n f'(u_h^n)) \phi_j \phi_i \, d\V{x}, & i=1, ..., N_{vi} \\
	0, & i=N_{vi}+1, ..., N_{v} 
\end{cases} , \quad j=1, ..., N_{v}
\\
\label{immatC}
& c_{ij} = \begin{cases}
	-  \int_\Omega u_h^n f'(u_h^n) \phi_j \phi_i \, d\V{x}, & i=1, ..., N_{vi} \\
	0, & i=N_{vi}+1, ..., N_{v} 
\end{cases} , \quad j=1, ..., N_{v}.
\end{align}
Substituting \eqref{rhs-imb} into \eqref{fem-dt}, we have
\beq
\label{fem-im}
(M -\Delta t_n B + \Delta t_n A) \V{u}^{n+1} = (M + \Delta t_n \, C) \V{u}^{n} + \Delta t_n \, \V{g}^{n+1}.
\eeq

\begin{thm}
\label{thm-im}
Assume $u_0(\V{x}) \ge 0$ and $g(\V{x},t) \ge 0$. The scheme \eqref{fem-im} preserves the nonnegativity
of the solution if the mesh satisfies 
\begin{equation}
\D_{acute} > \left (\frac{|\Omega|}{N_e}\right )^{\frac{2}{d}}
\frac{\max\limits_{\V{x}} \left | \left (f(u_h^n)+u_h^n f'(u_h^n)\right)^{-}\right |}{(d+1)(d+2)} 
\label{condt-im0}
\end{equation}
and the time step satisfies
\begin{align}
& \frac{\left (\frac{|\Omega|}{N_e}\right )^{\frac{2}{d}}}{(d+1)(d+2) \D_{acute} 
- \left (\frac{|\Omega|}{N_e}\right )^{\frac{2}{d}} \max\limits_{\V{x}} \left | \left (f(u_h^n)+u_h^n f'(u_h^n)\right)^{-}\right | }
\nn
\\
& \qquad \qquad \le \Delta t_n < \frac{1}{\max\limits_{\V{x}} \{\left (u_h^n f'(u_h^n)\right)^{+}, \left (f(u_h^n)+u_h^n f'(u_h^n)\right)^{+} \} } .
\label{condt-im}
\end{align}
\end{thm}

\begin{proof}
The proof is similar to that for Theorem~\ref{thm-em}.
A sufficient condition for $M+\Delta t_n C \ge 0$ is
\begin{equation}
1 - \Delta t_n u_h^n f'(u_h^n) \ge 0 \quad \text{or} \quad
\Delta t_n \le \frac{1}{\max\limits_{\V{x}} \left (u_h^n f'(u_h^n)\right)^{+} }.
\label{condt-im1}
\end{equation}

For the matrix $M-\Delta t_n B+\Delta t_n A$ on the left-hand side, the diagonal entries are
\bey
\nn
&& m_{ii} - \Delta t_n \, b_{ii} +  \Delta t_n \,  a_{ii} \\
\nn
&&\qquad \qquad =  \sum_{K \in \omega_i\cap \omega_j}  \left( \int_K (1 - \Delta t_n [f(u_h^n)+u_h^n f'(u_h^n)] ) \phi_i^2 \, d\V{x}
+ \Delta t_n |K| \;(\nabla \phi_{i})^T \DK  \nabla \phi_{i} \right),
\eey
and a sufficient condition for them to be nonnegative is 
\beq
\label{condt-im3}
\Delta t_n \le \frac{1}{\max\limits_{\V{x}} \left (f(u_h^n)+u_h^n f'(u_h^n)\right)^{+} }.
\eeq
The off-diagonal entries are given by
\bey
\nn
&& m_{ij} - \Delta t_n \, b_{ij} + \Delta t_n \,  a_{ij}  \\
\nn
&&\qquad \qquad =  \sum_{K \in \omega_i\cap \omega_j}  \left( \int_K (1 - \Delta t_n [f(u_h^n)+ u_h^nf'(u_h^n)] ) \phi_{j} \phi_{i} \, d\V{x}
+ \Delta t_n |K| \;(\nabla \phi_{i})^T \DK  \nabla \phi_{j} \right).
\eey
A sufficient condition for them to be nonpositive is (\ref{condt-im0}) and 
\begin{align}
\Delta t_n  \ge \frac{\left (\frac{|\Omega|}{N_e}\right )^{\frac{2}{d}}}{(d+1)(d+2) \D_{acute} 
- \left (\frac{|\Omega|}{N_e}\right )^{\frac{2}{d}} \max\limits_{\V{x}} \left | \left (f(u_h^n)+u_h^n f'(u_h^n)\right)^{-}\right | }.
\label{condt-im4}
\end{align}
For the diagonal dominance, we have 
\beq
\label{im-dom}
\sum_{j=1}^{N_v} (m_{ij}  -  \Delta t_n \, b_{ij} + \Delta t   \,  a_{ij})  
=  \int_\Omega (1-\Delta t_n [f(u_h^n)+u_h^n f'(u_h^n)]  \phi_{i} \, d\V{x},
\eeq
which is strictly positive if
\beq
\label{condt-im5}
\Delta t_n < \frac{1}{\max\limits_{\V{x}} \left (f(u_h^n)+u_h^n f'(u_h^n)\right)^{+} }.
\eeq

Combining the above conditions we obtain (\ref{condt-im}).
\end{proof}

\begin{rem}
\label{rem:3.2}
The condition (\ref{condt-im}) is comparable with (\ref{condt-em}), which effectively requires the mesh to be
acute. The conditions (\ref{condt-im}) and (\ref{condt-im0}) implies that
\[
(d+1)(d+2) \D_{acute} > \left (\frac{|\Omega|}{N_e}\right )^{\frac{2}{d}}
\left ( \max\limits_{\V{x}} \left | f(u_h^n)+u_h^n f'(u_h^n)\right | + 
\max\limits_{\V{x}} \left (u_h^n f'(u_h^n)\right)^{+} \right ).
\]
When the mesh is uniformly acute, the above condition, (\ref{condt-im}), and (\ref{condt-im0}) will essentially
become
\[
\mathcal{O}(1) \ge \mathcal{O}(h^2),\quad
\mathcal{O}(h^{2}) \le \Delta t_n < \mathcal{O}(1), \quad
\mathcal{O}(1) \ge \mathcal{O}(h^2) .
\]
They can be met easily when the mesh is sufficiently fine.
\qed
\end{rem}

\subsection{Hybrid explicit-implicit method (HEIM) }
\label{sec-HEIM}

We now consider two hybrid approximations using $u_h^n$ and $u_h^{n+1}$ for the reaction
term $u_h f(u_h)$ in \eqref{vecb}. 

\subsubsection{HEIM I}
\label{sec-heim1}

In this case, the reaction term is approximated by
\beq
\label{rhs-heim1}
\left. u_h f(u_h)\right |_{t_{n+1}} \approx u_h^{n+1} f(u_h^n). 
\eeq
This gives
\beq
\label{rhs-heim1b}
\tilde{\V{b}}(\V{u}^n,\V{u}^{n+1}) = B \, \V{u}^{n+1},
\eeq
where $B=(b_{ij})$ is given by
\beq
\label{heim1matB}
b_{ij} = \begin{cases}
	\int_\Omega f(u_h^n) \phi_j \phi_i \, d\V{x}, & i=1, ..., N_{vi} \\
	0, & i=N_{vi}+1, ..., N_{v} 
\end{cases} , \quad j=1, ..., N_{v}.
\eeq
Substituting the above into \eqref{fem-dt}, we have
\beq
(M - \Delta t_n B + \Delta t_n A) \V{u}^{n+1} = M \V{u}^{n} + \Delta t_n \, \V{g}^{n+1}.
\label{fem-heim1}
\eeq

Similar to Theorem~\ref{thm-im} for the scheme (\ref{fem-im}), we have the following theorem.

\begin{thm}
\label{thm-heim1} 
Assume $u_0(\V{x}) \ge 0$ and $g(\V{x},t) \ge 0$. The scheme \eqref{fem-im} preserves the nonnegativity
of the solution if the mesh satisfies 
\begin{equation}
\D_{acute} > \left (\frac{|\Omega|}{N_e}\right )^{\frac{2}{d}} \frac{\max\limits_{\V{x}} \left | f^{-}(u_h^n)\right |}{(d+1)(d+2)} 
\label{condt-heim10}
\end{equation}
and the time step satisfies
\begin{align}
\frac{\left (\frac{|\Omega|}{N_e}\right )^{\frac{2}{d}}}{(d+1)(d+2) \D_{acute}
- \left (\frac{|\Omega|}{N_e}\right )^{\frac{2}{d}} \max\limits_{\V{x}} \left | f^{-}(u_h^n)\right | } \le
\Delta t_n < \frac{1}{\max\limits_{\V{x}} f^{+}(u_h^n) } .
\label{condt-heim1}
\end{align}
\end{thm}

\begin{rem}
\label{rem3.3}
Conditions \eqref{condt-heim10} and \eqref{condt-heim1} imply 
\[
(d+1)(d+2) \D_{acute} >  \left (\frac{|\Omega|}{N_e}\right )^{\frac{2}{d}} \max\limits_{\V{x} \in \Omega} |f(u_h^n)|.
\]
This condition, \eqref{condt-heim10}, and \eqref{condt-heim1} can be met easily when the mesh is uniformly acute
and sufficiently fine.
\qed
\end{rem}

\subsubsection{HEIM II}
\label{sec-heim2}
In this case, the reaction term is approximated using the positive-negative part decomposition of $f(u_h^n)$, viz.,
\beq
\label{rhs-heim2}
\left. u_h f(u_h)\right |_{t_{n+1}} \approx u_h^{n+1} f^-(u_h^n) + u_h^{n} f^+(u_h^n).
\eeq
This approximation has been used by Qin et al. \cite{QDD15} for a finite difference solution of Nagumo's equation
(with $\D = I$). The approximation gives rise to
\beq
\label{rhs-heim2b}
\tilde{\V{b}}(\V{u}^n, \V{u}^{n+1}) = B \, \V{u}^{n+1} + C \, \V{u}^{n},
\eeq
where $B=(b_{ij})$ and $C=(c_{ij})$ are given by
\begin{align}
\label{matB2}
b_{ij} = \begin{cases}
	\int_\Omega f^{-}(u_h^n) \phi_j \phi_i \, d\V{x}, & i=1, ..., N_{vi} \\
	0, & i=N_{vi}+1, ..., N_{v} 
\end{cases} , \quad j=1, ..., N_{v}
\\
\label{matC}
c_{ij} = \begin{cases}
	\int_\Omega f^{+}(u_h^n) \phi_j \phi_i \, d\V{x}, & i=1, ..., N_{vi} \\
	0, & i=N_{vi}+1, ..., N_{v} 
\end{cases} , \quad j=1, ..., N_{v}.
\end{align}
Inserting \eqref{rhs-heim2b} into \eqref{fem-dt}, we have
\beq
(M -\Delta t_n B + \Delta t_n A) \V{u}^{n+1} = (M + \Delta t_n \, C) \V{u}^{n} + \Delta t_n \, \V{g}^{n+1}.
\label{fem-heim2}
\eeq
The matrix $M + \Delta t_n \, C$ is nonnegative. For the matrix $M -\Delta t_n B + \Delta t_n A$ on the left-hand side,
the diagonal entries are positive,
\[
m_{ii} - \Delta t_n \, b_{ii} +  \Delta t_n \,  a_{ii}  = \sum_{K \in \omega_i\cap \omega_j} 
\left( \int_K (1 - \Delta t_n f^{-}(u_h^n) ) \phi_i^2 \, d\V{x} + \Delta t_n |K| \;(\nabla \phi_{i})^T \DK
(\nabla \phi_{i} \right) > 0.
\]
The off-diagonal entries are given by
\[
m_{ij} - \Delta t_n \, b_{ij} + \Delta t_n \,  a_{ij}  =
\sum_{K \in \omega_i\cap \omega_j}  \left( \int_K (1 - \Delta t_n f^{-}(u_h^n) ) \phi_{j} \phi_{i} \, d\V{x}
+ \Delta t_n |K| \;(\nabla \phi_{i})^T \DK  \nabla \phi_{j} \right),
\]
which are nonpositive if
\[
\Delta t_n  \ge \frac{\left (\frac{|\Omega|}{N_e}\right )^{\frac{2}{d}}}{(d+1)(d+2) \D_{acute}
- \left (\frac{|\Omega|}{N_e}\right )^{\frac{2}{d}} \max\limits_{\V{x} \in \Omega} \left |f^-(u_h^n)\right |}.
\]
Similarly, for the diagonal dominance, we have 
\[
\sum_{j=1}^{N_v} (m_{ij}  -  \Delta t_n \, b_{ij} + \Delta t_n   \,  a_{ij})  
= \int_\Omega (1-\Delta t_n f^{-}(u_h^n) \phi_{i} \, d\V{x}
 >  0.
\]

Summarizing the above analysis, we obtain the following theorem.

\begin{thm}
\label{thm-heim2}
Assume $u_0(\V{x}) \ge 0$ and $g(\V{x},t) \ge 0$. The scheme \eqref{fem-heim2} preserves the nonnegativity
of the solution if the mesh satisfies
\begin{equation}
\D_{acute} > \left (\frac{|\Omega|}{N_e}\right )^{\frac{2}{d}} \frac{\max\limits_{\V{x}} \left | f^{-}(u_h^n)\right |}{(d+1)(d+2)} 
\label{condt-heim20}
\end{equation}
and the time step satisfies
\begin{align}
\Delta t_n  \ge \frac{\left (\frac{|\Omega|}{N_e}\right )^{\frac{2}{d}}}{(d+1)(d+2) \D_{acute}
- \left (\frac{|\Omega|}{N_e}\right )^{\frac{2}{d}} \max\limits_{\V{x} \in \Omega} \left |f^-(u_h^n)\right |} .
\label{condt-heim2}
\end{align}
\end{thm}

\begin{rem}
\label{rem3.4}
Conditions \eqref{condt-heim2} and \eqref{condt-heim20} are almost the same as \eqref{condt-heim1}
and \eqref{condt-heim10} except that $\Delta t_n$ has no upper bound for the current case.
They can be met when the mesh is uniformly acute and sufficiently fine.
\qed
\end{rem}

\subsection{Summary}

To summarize, we note that all approximations require that the mesh be at least acute
in the metric $\D^{-1}$ and the time step be bounded below and above (except HEIM II for which $\Delta t_n$
is bounded only below). When the mesh is uniformly acute in the metric $\D^{-1}$,
the conditions for $\Delta t_n$ essentially read as
\[
\mathcal{O}(h^2) \le \Delta t_n \le \mathcal{O}(1),
\]
which can be met when the mesh is sufficiently fine.

Moreover, by comparing the results in
Theorems~\ref{thm-em}, \ref{thm-im}, \ref{thm-heim1}, and \ref{thm-heim2}
in this section for Nagumo-type equations with Theorem 3.1 of \cite{LH13} for pure diffusion problems
we can see that the reaction term affects both the mesh and time step conditions. For the current situation
the mesh has to be at least acute in the metric $\D^{-1}$ and the time step is bounded below and above.
On the other hand, for pure diffusion problems it is only required that the mesh be nonobtuse
in the metric $\D^{-1}$ and the time step be bounded below.

\section{Lumping for the mass matrix and reaction term}
\label{sec-lumping}

In this section we consider the lumping technique for the mass matrix and the reaction term.
We first recall that the lumping technique is equivalent to approximating integrals with the nodal
numerical quadrature 
\[
\int_K v(\V{x}) d \V{x} \approx \frac{|K|}{d+1} \sum_{j=0}^d v(\V{x}_j),
\]
where $\V{x}_j$, $j = 0, ..., d$ denote the vertices of $K$.
Using this, the mass and reaction terms in (\ref{fem-form}) (with $v_h$ being replaced by $\phi_i$)
are approximated by
\begin{align*}
& \int_\Omega \frac{\partial u_h}{\partial t} \phi_i d \V{x}
= \sum_{K \in \Th} \int_K \frac{\partial u_h}{\partial t} \phi_i d \V{x}
\approx \frac{|\omega_i|}{d+1} \frac{d u_i}{d t} ,
\\
& \int_\Omega u_h f(u_h) \phi_i d \V{x}
= \sum_{K \in \Th} \int_K u_h f(u_h) \phi_i d \V{x}
\approx \frac{|\omega_i|}{d+1} u_i f(u_i) .
\end{align*}
Then the finite element equation with lumping is given by
\begin{equation}
\label{fem-lumping}
\bar{M} \frac{d \V{u}}{d t} + A \V{u} = \bar{\V{b}}(u_h) + \V{g},
\end{equation}
where $A$ and $\V{g}$ are given in (\ref{matA}) and (\ref{vecg}), respectively, and $\bar{M}$ (which is diagonal)
and $\bar{\V{b}}(u_h)$ are given by
\begin{align}
& \bar{m}_{ii} = \begin{cases} \frac{|\omega_i|}{d+1}, & \text{for} \quad i = 1, ..., N_{vi}
	\\ 0, & \text{for} \quad i = N_{vi} + 1, ..., N_v  \end{cases}
\label{matM-lumping}
\\
& \bar{b}_ i = \begin{cases} \frac{|\omega_i|}{d+1} u_i f(u_i) , & \text{for} \quad i = 1, ..., N_{vi}
	\\ 0, & \text{for} \quad i = N_{vi} + 1, ..., N_v .  \end{cases}
\label{vecb-lumping}
\end{align}
The backward Euler scheme for (\ref{fem-lumping}) is
\begin{equation}
\bar{M} \frac{\V{u}^{n+1}-\V{u}^{n}}{\Delta t_n} + A \V{u}^{n+1} = \tilde{\bar{\V{b}}}(\V{u}^n,\V{u}^{n+1}) + \V{g}^{n+1},
\label{fem-lumping-2}
\end{equation}
where $\tilde{\bar{\V{b}}}(\V{u}^n,\V{u}^{n+1})$ is an approximation to $\bar{\V{b}}(u_h)$ at $t= t_{n+1}$.
As in the previous section, we consider here four approximations to the reaction term.
Since the analysis is similar, we record below the sufficient conditions for the preservation of solution nonnegativity
without proof. 
\begin{itemize}
\item The \textbf{EM} approximation is
	\[
	\left. u_i f(u_i)\right |_{t_{n+1}} \approx u_i^n f(u_i^n) .
	\]
	The sufficient condition for the nonegativity preservation of the solution is that (a) the mesh satisfies
	ANOAC (\ref{anoac-1}) and (b) the time step satisfies
	\[
	\Delta t_n \le \frac{1}{\max_{i} \left | f^{-}(u_i^n)\right |} .
	\]
\item The \textbf{IM} approximation is
	\[
	\left. u_i f(u_i)\right |_{t_{n+1}} \approx u_i^{n+1} f(u_i^{n+1}) \approx u_i^{n+1} (f(u_i^n)+u_i^n f'(u_i^n))
	- u_i^n f'(u_i^n) u_i^n .
	\]
	The sufficient condition for the nonegativity preservation of the solution is that (a) the mesh satisfies
	ANOAC (\ref{anoac-1}) and (b) the time step satisfies
	\[
	\Delta t_n \le \frac{1}{\max_{i} \{ (f(u_i^n)+u_i^n f'(u_i^n))^{+}, (u_i^n f'(u_i^n))^{+}\}} .
	\]
\item The \textbf{HEIM I} approximation is
	\[
	\left. u_i f(u_i)\right |_{t_{n+1}} \approx u_i^{n+1} f(u_i^{n}) .
	\]
	The sufficient condition for the nonegativity preservation of the solution is that (a) the mesh satisfies
	ANOAC (\ref{anoac-1}) and (b) the time step satisfies
	\begin{equation}
	\Delta t_n < \frac{1}{\max_{i} f^{+}(u_i^n)} .
	\label{heim1-lumping}
	\end{equation}
\item The \textbf{HEIM II} approximation is
	\[
	\left. u_i f(u_i)\right |_{t_{n+1}} \approx u_i^{n+1} f^{-}(u_i^{n}) + u_i^{n} f^{+}(u_i^{n}) .
	\]
	The sufficient condition for the nonegativity preservation of the solution is that the mesh satisfies
	ANOAC (\ref{anoac-1}).

\end{itemize}

By comparing the above results with those in the previous section, we can see that the lumping of
the mass and reaction terms improves both the mesh and time step conditions. In the current situation,
it is sufficient to require that \textit{the mesh be nonobtuse} (instead of acute or uniformly acute). Moreover,
\textit{the time step does not need to be bounded below}. Once again, HEIM II gives the weakest condition,
which does not require the time step be bounded above either.

\section{Preservation of boundedness}
\label{sec-boundedness}

The analysis in the previous two sections can also be applied to preservation of solution boundedness. 
We take Nagumo's equation as an example. It has $f(u) = (1-u)(u-a)$ for $a \in (0,1)$ and
the upper bound $u \le 1$.

We first consider the explicit method (\ref{fem-em}). Notice that $u_h \le 1$ is equivalent to $v_h \ge 0$,
where $v_h = 1 - u_h$. Inserting $\V{u}^{n} = \V{e} - \V{v}^{n}$ and $\V{u}^{n+1} = \V{e} - \V{v}^{n+1}$
into (\ref{fem-em}), where $\V{e} = [1, ..., 1]^T$, and using the properties of matrices $A$ and $C$, we obtain
\begin{equation}
(M + \Delta t_n A) \V{v}^{n+1} = (M + \Delta t_n \tilde{C} ) \V{v}^{n} + \Delta t_n \tilde{\V{g}}^{n+1},
\label{v-1}
\end{equation}
where
\begin{align*}
& \tilde{c}_{i,j} = \begin{cases} 
 \int_\Omega \phi_i \phi_j u_h^n (a-u_h^n ) d \V{x}, & i=1, ..., N_{vi} \\
0, & i=N_{vi}+1, ..., N_{v} 
\end{cases}
j = 1, ..., N_v
\\
& \tilde{g}_i^{n+1} = \begin{cases}
	0, & i=1, ..., N_{vi} \\
	1-g(\V{x}_i,t_{n+1}), & i=N_{vi}+1, ..., N_{v} .
\end{cases}
\end{align*}
Like Theorem~\ref{thm-em}, we can show that $\V{v}^{n+1} \ge 0$ when $\V{v}^{n}\ge 0$,
$\tilde{\V{g}}^{n+1} \ge 0$, and
\[
\frac{\left (\frac{|\Omega|}{N_e}\right )^{\frac{2}{d}}}{(d+1)(d+2)\D_{acute}} \le \Delta t_n
\le \frac{1}{\max\limits_{\V{x} }|(u_h^n (a-u_h^n))^{-}|} .
\]
Combining this with Theorem~\ref{thm-em} and recalling that $\V{v}^{n+1} \ge 0$ implies $\V{u}^{n+1} \le 1$,
we obtain the following theorem.

\begin{thm}
\label{thm-em-2}
Assume $0\le u_0(\V{x})\le 1$ and $0 \le g(\V{x},t) \le 1$. The scheme \eqref{fem-em} for Nagumo's equation
preserves the nonnegativity and boundedness ($u \le 1$) of the solution
of the continuous problem if the mesh satisfies ANOAC (\ref{anoac-1}), i.e., $\D_{acute} \ge 0$,
and the time step satisfies
\beq
\label{condt-em-2}
\frac{\left (\frac{|\Omega|}{N_e}\right )^{\frac{2}{d}}}{(d+1)(d+2)\D_{acute}} \le \Delta t_n
\le \frac{1}{\max\limits_{\V{x} }\{|f^{-}(u_h^n)|,|(u_h^n (a-u_h^n))^{-}|\} }, \quad n = 0, 1, ...
\eeq
\end{thm}

By comparing Theorem~\ref{thm-em} and the above theorem we can see that the preservation
of both solution boundedness and nonnegativity imposes a stronger condition on the allowable
maximum time step but other conditions stay the same. Interestingly, this is also true for
other schemes. Indeed, using a similar analysis we can show that the time step condition
for the preservation of solution boundedness and nonnegativity for the implicit method (\ref{fem-im})
is 
\begin{align}
& \frac{\left (\frac{|\Omega|}{N_e}\right )^{\frac{2}{d}}}{(d+1)(d+2) \D_{acute} 
- \left (\frac{|\Omega|}{N_e}\right )^{\frac{2}{d}} \max\limits_{\V{x}} \left | \left (f(u_h^n)+u_h^n f'(u_h^n)\right)^{-}\right | }
\le \Delta t_n
\nn
\\
& \quad < \frac{1}{\max\limits_{\V{x}} \{\left (u_h^n f'(u_h^n)\right)^{+},
\left (f(u_h^n)+u_h^n f'(u_h^n)\right)^{+}, (u_h^n-a+u_h^nf'(u_h^n))^{+} \} } .
\label{condt-im-2}
\end{align}
The time step condition for the HEIM I method (\ref{fem-heim1}) is 
\begin{align}
\frac{\left (\frac{|\Omega|}{N_e}\right )^{\frac{2}{d}}}{(d+1)(d+2) \D_{acute}
- \left (\frac{|\Omega|}{N_e}\right )^{\frac{2}{d}} \max\limits_{\V{x}} \left | f^{-}(u_h^n)\right | } \le
\Delta t_n < \frac{1}{\max\limits_{\V{x}} \{ f^{+}(u_h^n), (u_h^n-a)^{+}\} } ,
\label{condt-heim1-2}
\end{align}
while that for the HEIM II method (\ref{fem-heim2}) is given by
\begin{align}
& \frac{\left (\frac{|\Omega|}{N_e}\right )^{\frac{2}{d}}}{(d+1)(d+2) \D_{acute}
- \left (\frac{|\Omega|}{N_e}\right )^{\frac{2}{d}} \max\limits_{\V{x}} \left |f^-(u_h^n)\right |} 
\notag \\
& \qquad \qquad \le \Delta t_n  \le \frac{1}{\max\limits_{\V{x}} ( (u_h^n-a)^{-} + u_h^n (u_h^n-a)^{+})^{+} } .
\label{condt-heim2-2}
\end{align}

The situations with lumping for the mass matrix and reaction term can also be analyzed similarly.
The results are omitted here to save space.

\section{Numerical results}
\label{sec-ex}

In this section we present three numerical examples in two dimensions to demonstrate the effectiveness
of the conditions for preservation of nonnegativity and boundedness for the Nagumo equation with $f = (1-u)(u-0.1)$.
We consider four different meshes shown in Fig.~\ref{meshes}, including meshIso, meshAcute,
mesh45, and mesh135. MeshIso is a Delaunay mesh generated using the free C++ code BAMG \cite{bamg},
meshAcute is generated by splitting each subsquare into 8 acute triangles \cite{BKK09},
mesh45 is generated by splitting each subsquare into two right triangles with the hypotenuse
aligning in the direction of 45 degree, and mesh135 is similar to mesh45 except the hypotenuse
is in the direction of 135 degree.  Results with and without lumping for the mass matrix and the reaction term are presented.

\begin{figure}[hbt]
\begin{center}
\begin{minipage}[b]{2in}
\includegraphics[width=2in]{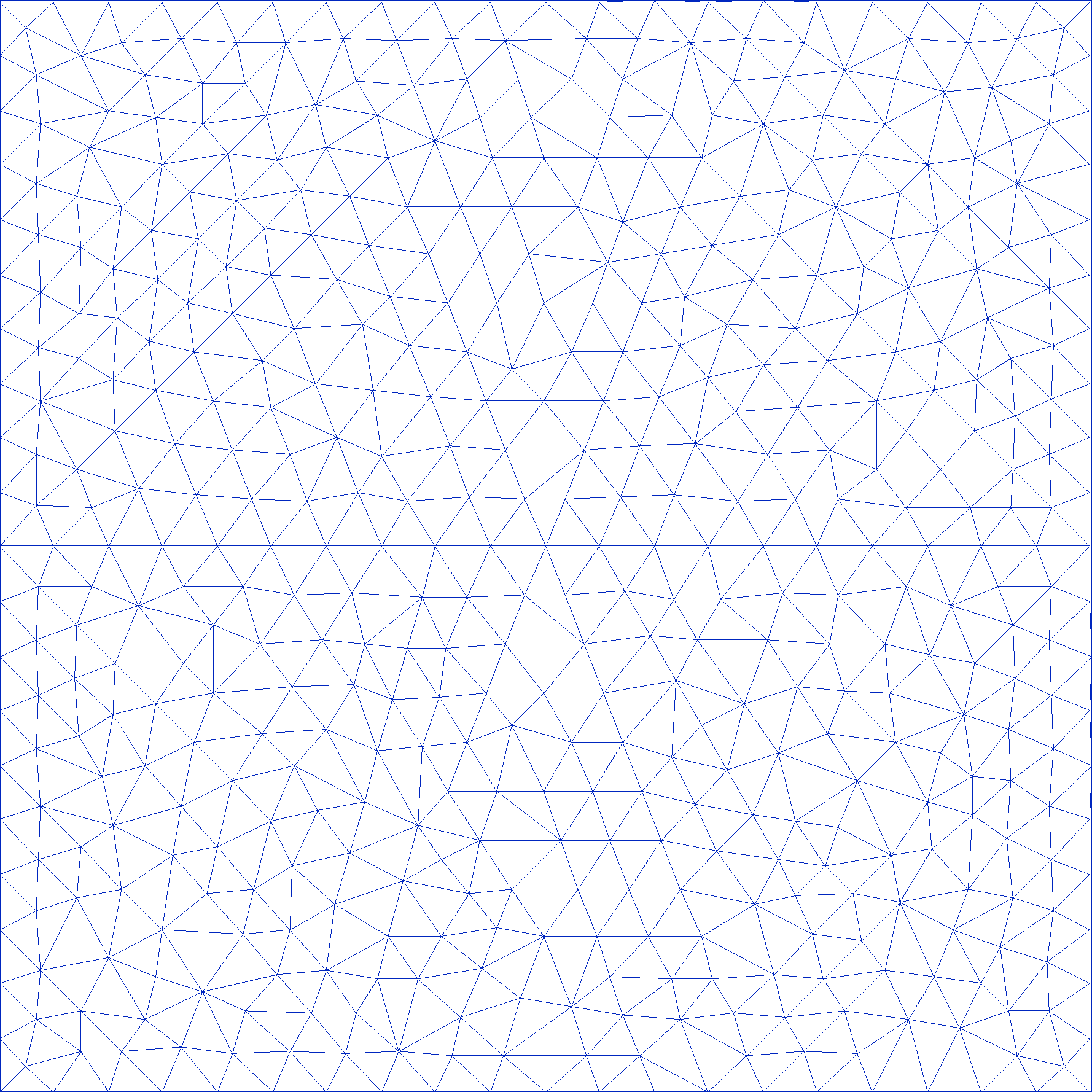}
\centerline{(a): meshIso, $N_e=964$}
\end{minipage}
\begin{minipage}[b]{2in}
\includegraphics[width=2in]{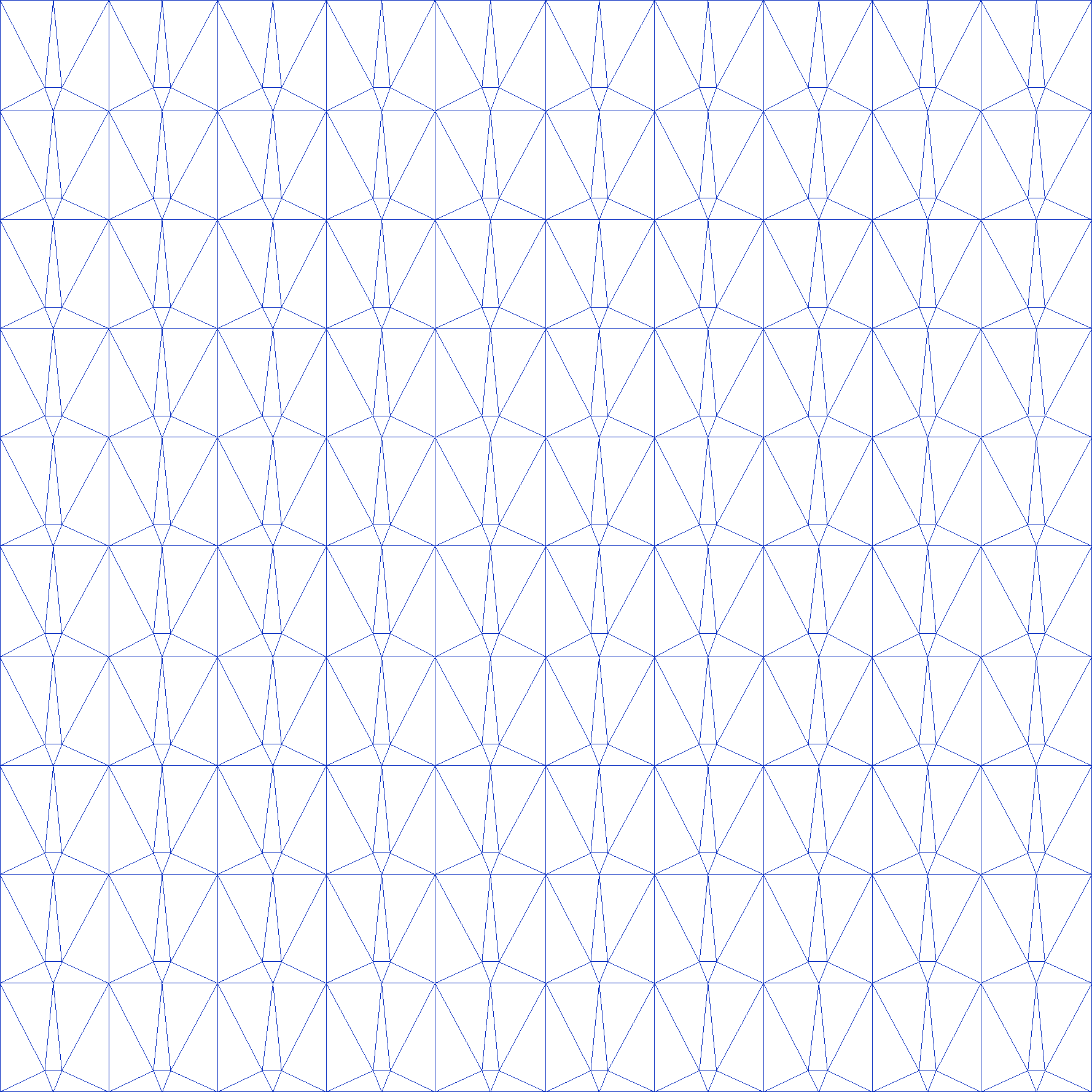}
\centerline{(b): meshAcute, $N_e=800$}
\end{minipage}
\\
\vspace{5mm}
\begin{minipage}[b]{2in}
\includegraphics[width=2in]{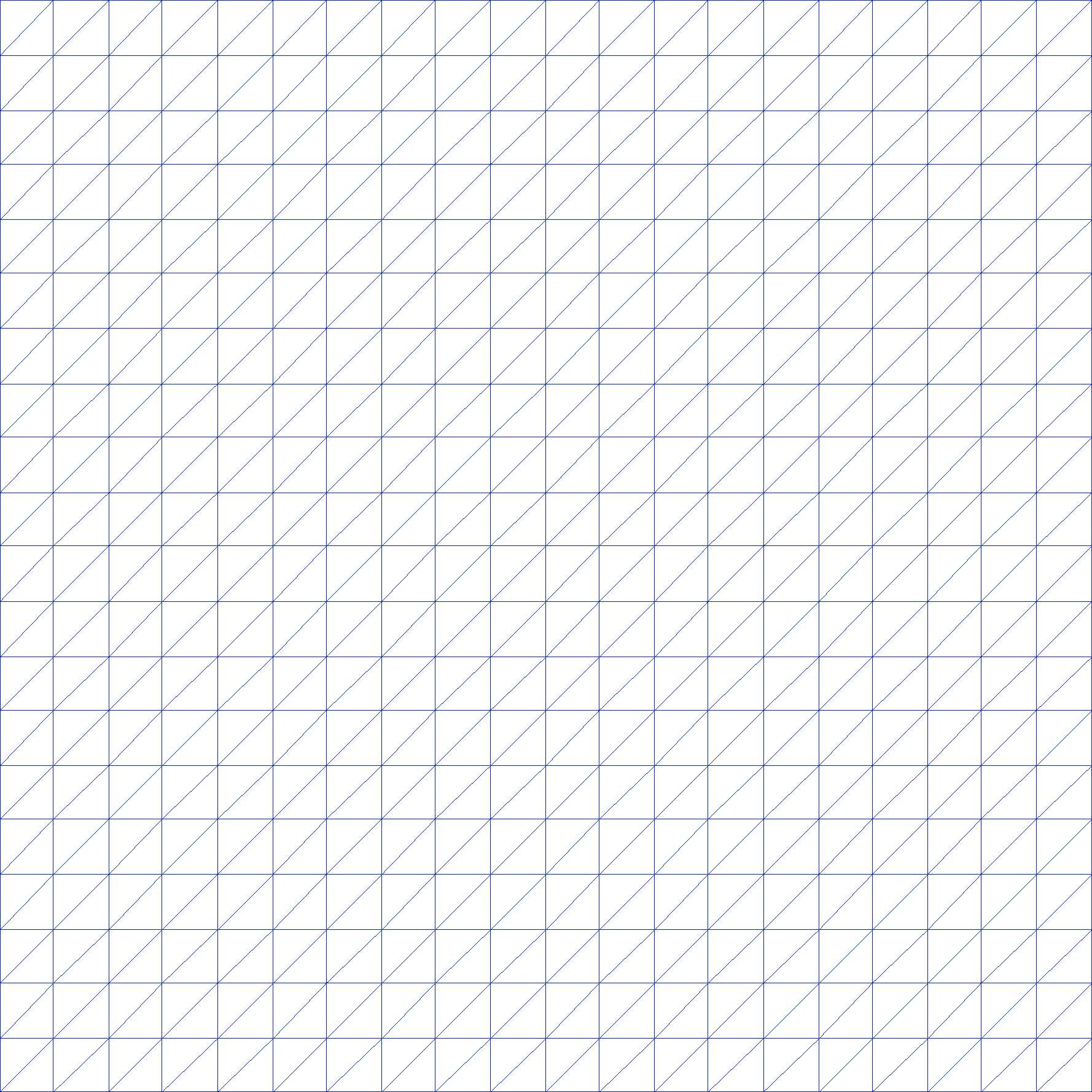}
\centerline{(c): mesh45, $N_e=800$}
\end{minipage}
\begin{minipage}[b]{2in}
\includegraphics[width=2in]{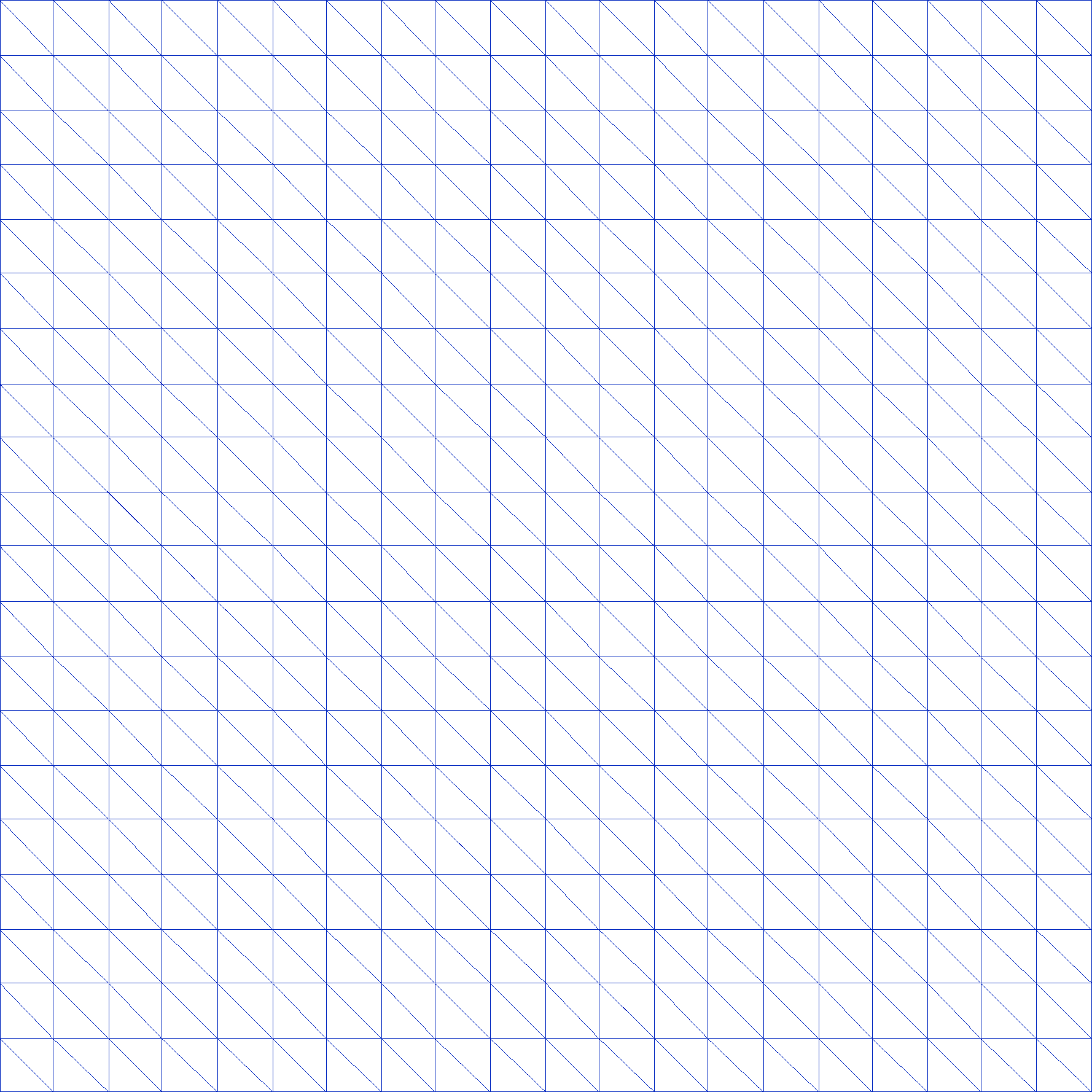}
\centerline{(d): mesh135, $N_e=800$}
\end{minipage}
\end{center}
\caption{Different meshes used in the numerical computations.}
\label{meshes}
\end{figure}     

The first example is an isotropic problem with a known exact solution.
The second example is anisotropic but homogeneous, that is, the diffusion matrix is constant but has different eigenvalues. The last example is anisotropic and inhomogeneous where the diffusion matrix not only has different eigenvalues and but also varies in location. For simplicity, we only consider the cases when the diffusion matrix is time independent.

\begin{exam}
\label{ex1}
The first example is in the form of \eqref{nagumo} with $\Omega=(-100,100)\times (-100,100)$, $\D=I$, 
and $g(x,y,t)$ and $u_0(x,y)$ are given such that the equation \eqref{nagumo} has the exact solution
\beq
\label{uexact}
u(x,y,t) = \frac{ e^{0.5(x+y)+(0.5-0.1)t} }{e^{0.5(x+y)+(0.5-0.1)t} + 2},
\eeq	
which satisfies $0 \le u \le 1$.

We first verify the convergence order of the schemes.
The $L^2$-norm of the error using the four methods (EM, IM, HEIM1, HEIM2) with mesh45
is scaled by the square root of the area of the domain. 
For the convergence in time, $T=10$ and a fixed, fine mesh with $N_e=2,000,000$ are used in the computations.
The results are shown in Fig.~\ref{ex1-errordtF}, which clearly show first order convergence in time. It is interesting
to point out that  HEIM I has the smallest error while HEIM II has the largest one.

For the convergence in space, $T=0.25$ and $\Delta t= 10^{-4}$ are used in the computations.
A small time step is used so that the temporal discretization error stays at a negligible level. The results are shown
in Fig.~\ref{ex1-errordxF}. They demonstrate a second order convergence in space (i.e., $\mathcal{O}(1/N_e)$).
One may notice that the error for the four approximations is indistinguishable
since they have similar spatial discretization errors. 

We remark that the schemes show a similar convergence behavior on other meshes.
The results are not presented here to save space. Moreover,
the diffusion for this example is isotropic and the ANOAC condition (that is, $D_{acute} \ge 0$) reduces to the conventional nonobtuse-angle condition. We have $D_{acute}=- 0.277$ for meshIso
shown in Fig.~\ref{meshes}(a) (which has obtuse triangles), $D_{acute}=0$ for both mesh45 and mesh135,
and $D_{acute}=0.0116$ for meshAcute. Thus, only meshAcute satisfies the mesh conditions
in Theorems~\ref{thm-em}, \ref{thm-im}, \ref{thm-heim1}, and \ref{thm-heim2} which essentially require the mesh
to be acute. Nevertheless, the numerical solutions obtained from all the four methods on all the meshes
do not violate the solution nonnegativitiy when the mesh is sufficiently fine. This is because those are
only sufficient conditions. The numerical solution is guaranteed to satisfy the nonnegativity when the conditions
are met and there is no such guarantee otherwise. The next example will demonstrate this.

\begin{figure}[hbt]
\centering
\includegraphics[width=2.5in]{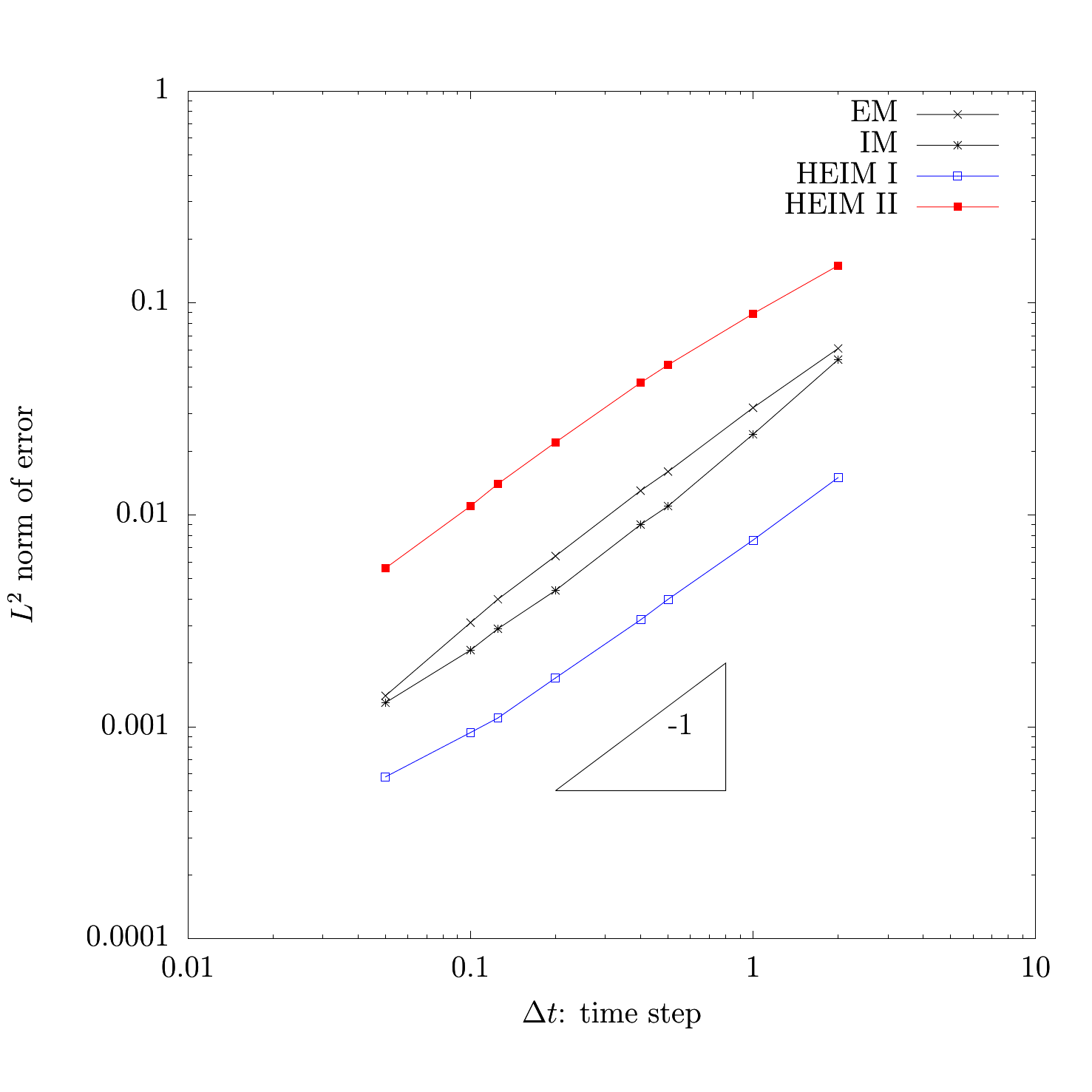}
\caption{Example \ref{ex1}. The $L^2$-norm of the error for different time steps using mesh45 with $N_e=2,000,000$ and $T=10.0$.}
\label{ex1-errordtF}
\end{figure}

\begin{figure}[hbt]
\centering
\includegraphics[width=2.5in]{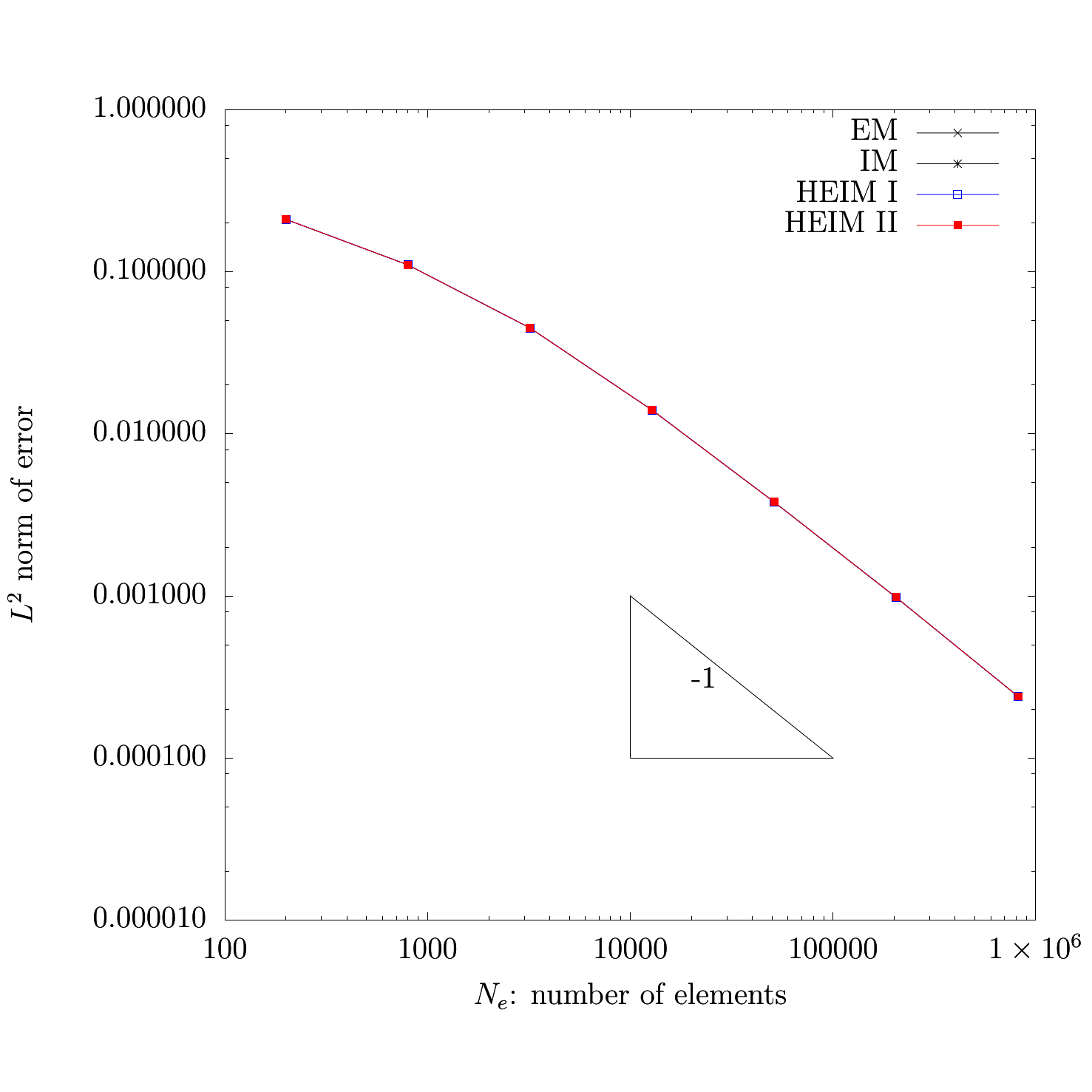}
\caption{Example \ref{ex1}. The $L^2$-norm of the error for different mesh sizes using mesh45 with $\Delta t=1.0 \times 10^{-4}$ and $T=0.25$.}
\label{ex1-errordxF}
\end{figure}

\end{exam}

\begin{exam}
\label{ex2}
This example has the same boundary and initial conditions as Example \ref{ex1} but 
its domain is $\Omega=(-100,100) \times (-170,170)$ and the diffusion matrix is anisotropic
and homogeneous (i.e., not changing with location),
\beq
\D = \frac{1}{4} \left[ \begin{array}{cc} 203 & 199 \sqrt{3} \\ 199 \sqrt{3} & 601 \end{array} \right].
\eeq
This anisotropic diffusion matrix has eigenvalues $\lambda_1=200$ and $\lambda_2=1$, and the eigenvector corresponding to the principal eigenvalue $\lambda_1$ is in the direction of $60$ degree. There is no exact solution available; however, the solution should lie in the interval $(0,1)$ according to the  preservation of nonnegativity and boundedness. 

Notice that the domain is no longer square. Indeed,  the hypotenuse side of each triangle element
in mesh45 is now in the direction of $59.5$ degree instead of $45$ degree as in the previous example.
For convenience and without confusion, we still denote the mesh as mesh45. The situation is similar for mesh135. 
Table~\ref{ex2-Dacute} shows the values of $D_{acute}$ and $D_{acute,ave}$ for the four meshes,
where $D_{acute,ave}$ is an averaged indicator of $D_{acute}$ defined as
\beq
\label{Dacute-ave}
\D_{acute,ave} = \frac{1}{N} \sum\limits_{K \in \Th} \left (\frac{|\Omega|}{N_e}\right )^{\frac{2}{d}}
\min\limits_{\substack{i, j = 0, ..., d\\ i \neq j}} (-(\nabla \phi_i)^T \DK \nabla \phi_j).
\eeq
As can be seen from Table~\ref{ex2-Dacute}, the values of $\D_{acute}$ are negative for all meshes
except mesh45. Thus, none of the meshIso, mesh135 and meshAcute satisfies the anisotropic
acute angle condition. On the other hand, the elements of mesh45 are closely aligned along the principal
diffusion direction and the positive $\D_{acute}$ implies that the mesh satisfies the anisotropic
acute angle condition.

The numerical solutions  at $T=40$ obtained with time step $\Delta t=0.1$ for the four meshes are shown in Fig.~\ref{ex2-soln}. 
Table \ref{ex2-DMP} lists the minimum and maximum values in the numerical solutions, denoted by $u_{min}$ and $u_{max}$, respectively. The numerical solutions from all meshes except mesh45 has negative values (undershoot) and values larger than 1 (overshoot), which violate the nonnegativity and upper boundedness. The regions of undershoot and overshoot are displayed in Fig.~\ref{ex2-soln} as empty white areas. Only the solutions obtained from mesh45 are guaranteed to stay within [0,1]. 

It is instructive to see that the upper bound of the time step is $\Delta t_{ub}=10$ in \eqref{condt-em} for EM method,
$\Delta t_{ub}=3.3$ in \eqref{condt-im} for IM, $\Delta t_{ub}=4.9$ in \eqref{condt-heim1} for HEIM I,
and $\Delta t_{ub}=\infty$ in \eqref{condt-heim2} for HEIM II (i.e., no upper bound). As mentioned in the previous
sections, they are not really a restriction in practical computation.

 The results using lumped mass and reaction terms are presented in Table \ref{ex2-lump}. With lumping technique, the results are close to those without lumping (Table \ref{ex2-DMP}) but with slight improvements.
For example, using HEIM II method, $u_{min}$ and $u_{max}$ are improved from $-0.0092$ and $1.0041$
without lumping to $-0.0088$ and $1.0035$ with lumping, respectively. 
Recall that the situations with lumping have less restrictive requirements on time step but the mesh condition is
almost the same as for those without lumping.

\begin{table}[hbt]
\caption{Measures for satisfaction of the anisotropic non-obtuse angle condition for Example \ref{ex2}.}
\vspace{2pt}
\centering
\begin{tabular}{c|cccc}
\hline \hline
Mesh & meshIso & mesh45 & mesh135 & meshAcute  \\
\hline \hline
 $N_e$ & 51,190 & 51,200 & 51,200 & 51,200 \\
\hline
 $\D_{acute}$ & -1.7e+2 & \textbf{5.3e-2} & -4.3e+1 & -2.0e+2 \\
\hline 
 $\D_{acute,ave}$ & -2.3e+1 & \textbf{5.3e-2} & -4.3e+1 & -5.8e+1 \\
\hline \hline
\end{tabular} \\
\label{ex2-Dacute}
\end{table}

\begin{figure}
\begin{center}
\begin{minipage}[b]{2in}
\includegraphics[width=2in]{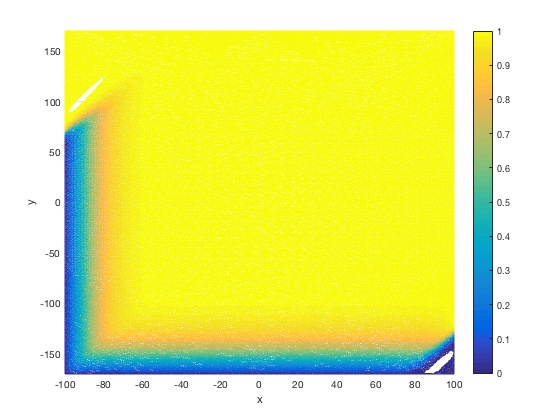}
\centerline{(a): meshIso, $N_e=51,190$}
\end{minipage}
\begin{minipage}[b]{2in}
\includegraphics[width=2in]{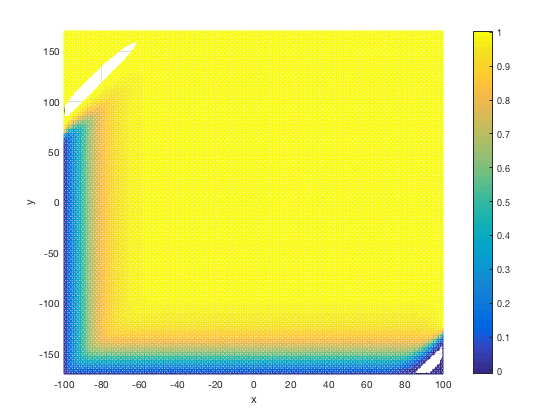}
\centerline{(b): meshAcute, $N_e=51,200$}
\end{minipage}
\\
\vspace{5mm}
\begin{minipage}[b]{2in}
\includegraphics[width=2in]{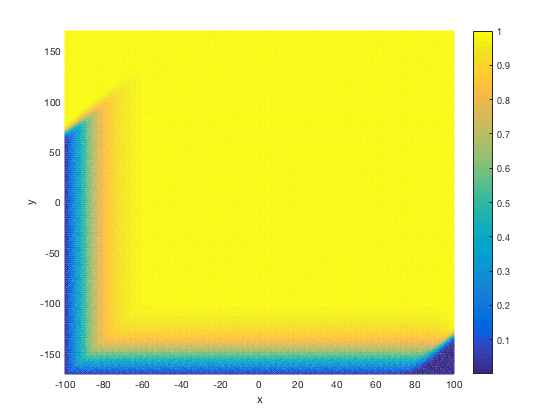}
\centerline{(c): mesh45, $N_e=51,200$}
\end{minipage}
\begin{minipage}[b]{2in}
\includegraphics[width=2in]{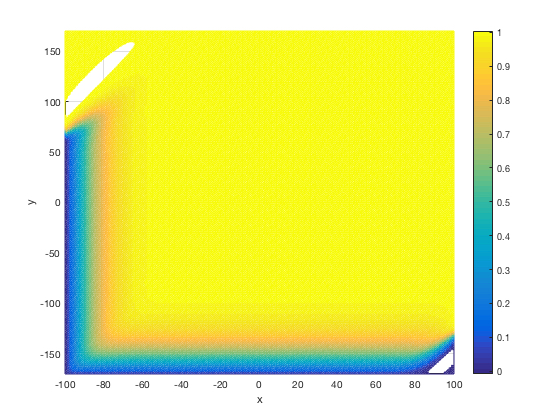}
\centerline{(d): mesh135, $N_e=51,200$}
\end{minipage}
\end{center}
\caption{Example \ref{ex2}. Filled contour plot of the numerical solutions at $T = 40$ obtained from different meshes with $\Delta t = 0.1$. The regions of undershoot and overshoot are displayed as empty white areas.}
\label{ex2-soln}
\end{figure}   

\begin{table}[hbt]
\caption{Numerical solutions obtained using different meshes and methods for Example \ref{ex2}.}
\vspace{2pt}
\centering
\begin{tabular}{cc|cccc}
\hline \hline
Mesh & solution & EM & IM & HEIM I & HEIM II  \\
\hline \hline
 meshIso & $u_{min}$ & -3.1e-3 & -3.1e-3 & -3.1e-3 & -3.1e-3 \\
 $N_e=51,190$ & $u_{max}$ & 1.0004 & 1.0003 & 1.0004 & 1.0003 \\
\hline
 mesh45 & $u_{min}$ & 0 & 0 & 0 & 0 \\
 $N_e=51,200$ & $u_{max}$ & 1 & 1 & 1 & 1 \\
\hline 
 mesh135 & $u_{min}$ & -8.1e-3 & -8.1e-3 & -8.1e-3 & -8.1e-3 \\
 $N_e=51,200$ & $u_{max}$ & 1.0041 & 1.0040 & 1.0040 & 1.0040 \\
\hline
 meshAcute & $u_{min}$ & -9.2e-3 & -9.3e-3 & -9.3e-3 & -9.2e-3 \\
 $N_e=51,200$ & $u_{max}$ & 1.0041 & 1.0040 & 1.0041 & 1.0040 \\
\hline \hline
\end{tabular} \\
\label{ex2-DMP}
\end{table}  

\begin{table}[hbt]
\caption{Numerical solutions obtained using different meshes and methods for Example \ref{ex2} with lumped mass and reaction terms.}
\vspace{2pt}
\centering
\begin{tabular}{cc|cccc}
\hline \hline
Mesh & solution & EM & IM & HEIM I & HEIM II  \\
\hline \hline
 meshIso & $u_{min}$ & -2.8e-3 & -2.9e-3 & -2.9e-3 & -2.8e-3 \\
 $N_e=51,190$ & $u_{max}$ & 1.0003 & 1.0002 & 1.0002 & 1.0002 \\
\hline
 mesh45 & $u_{min}$ & 0 & 0 & 0 & 0 \\
 $N_e=51,200$ & $u_{max}$ & 1 & 1 & 1 & 1 \\
\hline 
 mesh135 & $u_{min}$ & -8.0e-3 & -8.0e-3 & -8.0e-3 & -8.0e-3 \\
 $N_e=51,200$ & $u_{max}$ & 1.0037 & 1.0037 & 1.0037 & 1.0037 \\
\hline
 meshAcute & $u_{min}$ & -8.8e-3 & -8.8e-3 & -8.8e-3 & -8.8e-3 \\
 $N_e=51,200$ & $u_{max}$ & 1.0035 & 1.0034 & 1.0035 & 1.0034 \\
\hline \hline
\end{tabular} \\
\label{ex2-lump}
\end{table}

\end{exam}

\begin{exam}
\label{ex3}
In this example, we consider a diffusion matrix that is anisotropic and inhomogeneous (i.e., location dependent). Specifically, we choose the diffusion matrix in the form
\beq
\D = \left[ \begin{array}{ll} \cos \theta & -\sin \theta \\ \sin \theta & \cos \theta \end{array} \right]
	 \left[ \begin{array}{ll} 200 & 0 \\ 0 & 1 \end{array} \right] 	
	 \left[ \begin{array}{ll} \cos \theta & \sin \theta \\ -\sin \theta & \cos \theta \end{array} \right],
\eeq
where $\theta=\theta(x,y)$ is the angle of the tangential direction at point $(x,y)$ along concentric circles centered at (0,0).
The domain and the boundary and initial conditions are chosen the same as in Example \ref{ex1}.
 The solution also preserves nonnegativity and boundedness and should stay within $[0,1]$.

For this example, none of meshIso, mesh45, mesh135, and meshAcute satisfies the corresponding mesh
conditions, as can be seen from the values of $\D_{acute}$ in Table~\ref{ex3-Dacute}. For the purpose of comparison, another mesh, denoted by meshDMP, is generated according to the metric tensor $\MK=\DK^{-1}$ as
proposed in \cite{LH10}. 
The elements in meshDMP are made to be aligned along the principal diffusion direction as much as possible,
as shown in Fig.~\ref{ex3-meshDMP}. However, it is impossible to force every element in the square domain to be aligned
along the concentric direction. Some elements that are not aligned well will violate the mesh condition
and lead to negative $\D_{acute}$ value ($\D_{acute}=-1500$). On the other hand, $\D_{acute,ave}=1.1$,
which indicates that most elements in the mesh satisfy the mesh condition. In this sense, meshDMP can be
considered to closely satisfy the anisotropic acute angle condition for the given diffusion matrix. 

Table \ref{ex3-DMP} shows the minimum and maximum values of the numerical solutions obtained
from the five meshes. Both undershoot and overshoot are observed in the solutions obtained with
meshIso, mesh45, mesh135, and meshAcute whereas the solutions obtained from meshDMP stay within [0, 1]. 
The numerical solutions at $T=40$ obtained using different meshes and $\Delta t=0.1$ are shown
in Fig.~\ref{ex3-soln2} where the regions of undershoot and overshoot are displayed as empty white areas.
The numerical solution obtained from meshDMP with $N_e=51,146$ does not have undershoot and overshoot.
When the mesh is being refined, undershoot and overshoot vanish from the numerical solution
for meshIso with $N_e=591,506$ but are still present in the solutions obtained with much finer mesh45, mesh135, or meshAcute.

 The results with lumped mass and reaction terms are shown in Table \ref{ex3-lump}. Like Example \ref{ex2},
there are slight improvements by using lumping technique. One notable improvement is that $u_{min}$ for mesh135 is 0,
that is, no undershoot is observed for results obtained using mesh135 with lumped mass and reaction terms. 

\begin{table}[hbt]
\caption{Measures for the satisfaction of the anisotropic non-obtuse angle condition for Example \ref{ex3}.}
\vspace{2pt}
\centering
\begin{tabular}{c|ccccc}
\hline \hline
Mesh & meshIso & mesh45 & mesh135 & meshAcute  & meshDMP \\
\hline \hline
 $N_e$ & 51,910 & 51,200 & 51,200 & 51,200 & 51,146 \\
\hline
 $\D_{acute}$ & -1.2e+2 & -5.0e+1 & -5.0e+1 & -2.6e+2 & -1.5e+3 \\
\hline 
 $\D_{acute,ave}$ & -1.8e+1 & -2.3e+1 & -2.3e+1 & -5.0e+1 & \textbf{1.1} \\
\hline \hline
\end{tabular} \\
\label{ex3-Dacute}
\end{table}

\begin{figure}[hbt]
\centering
\includegraphics[width=2.5in]{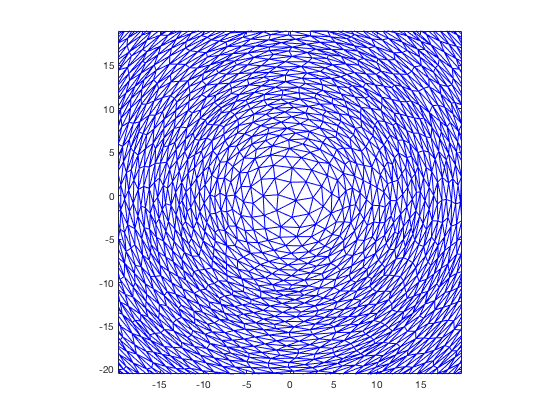}
\caption{Example \ref{ex3}. Scaleup view of meshDMP at (0,0).}
\label{ex3-meshDMP}
\end{figure} 

\begin{table}[hbt]
\caption{Numerical solutions obtained from different meshes and methods for Example \ref{ex3}.}
\vspace{2pt}
\centering
\begin{tabular}{cc|cccc}
\hline \hline
Mesh & solution & EM & IM & HEIM I & HEIM II  \\
\hline \hline
 meshIso & $u_{min}$ & -1.3e-3 & -1.2e-3 & -1.3e-3 & -1.4e-3 \\
 $N_e=51,910$ & $u_{max}$ & 1.0129 & 1.0129 & 1.0129 & 1.0129 \\
\hline
 mesh45 & $u_{min}$ & -1.4e-2 & -1.4e-2 & -1.4e-2 & -1.5e-2 \\
 $N_e=51,200$ & $u_{max}$ & 1.0101 & 1.0101 & 1.0101 & 1.0101 \\
\hline 
 mesh135 & $u_{min}$ & -4.4e-6 & -3.0e-6 & -3.8e-6 & -8.4e-6 \\
 $N_e=51,200$ & $u_{max}$ & 1.0082 & 1.0082 & 1.0082 & 1.0082 \\
\hline
 meshAcute & $u_{min}$ & -1.0e-2 & -1.0e-2 & -1.0e-2 & -1.1e-2 \\
 $N_e=51,200$ & $u_{max}$ & 1.0165 & 1.0165 & 1.0165 & 1.0164 \\
\hline
 meshDMP & $u_{min}$ & 0 & 0 & 0 & 0 \\
 $N=51,146$ & $u_{max}$ & 1 & 1 & 1 & 1 \\
\hline \hline
\end{tabular} \\
\label{ex3-DMP}
\end{table}  

\begin{figure}[hbt]
\begin{center}
\begin{minipage}[b]{2in}
\includegraphics[width=2in]{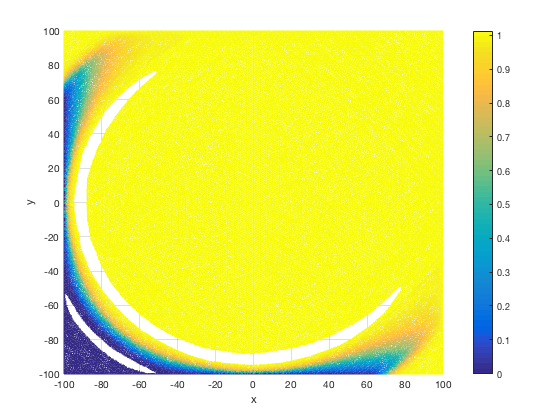}
\centerline{(a): meshIso, $N_e=51,910$}
\end{minipage}
\hspace{5mm}
\begin{minipage}[b]{2in}
\includegraphics[width=2in]{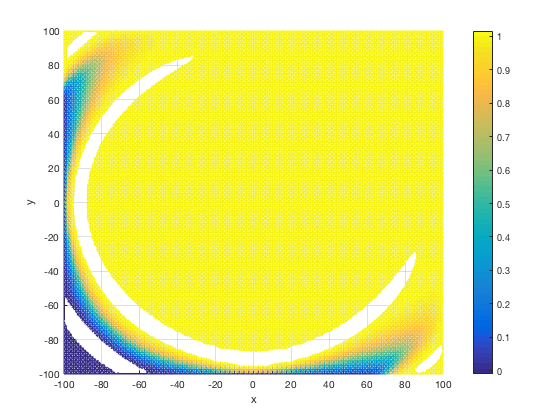}
\centerline{(b): meshAcute, $N_e=51,200$}
\end{minipage}
\\
\vspace{5mm}
\begin{minipage}[b]{2in}
\includegraphics[width=2in]{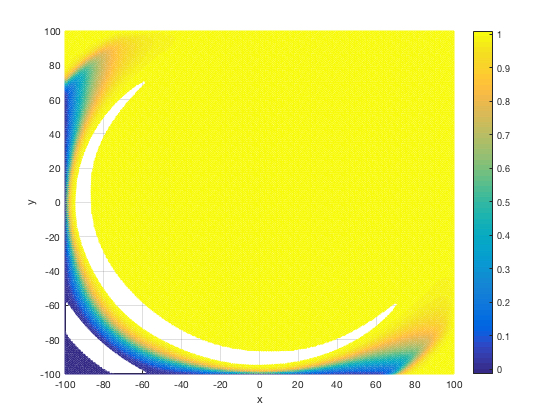}
\centerline{(c): mesh45, $N_e=51,200$}
\end{minipage}
\hspace{5mm}
\begin{minipage}[b]{2in}
\includegraphics[width=2in]{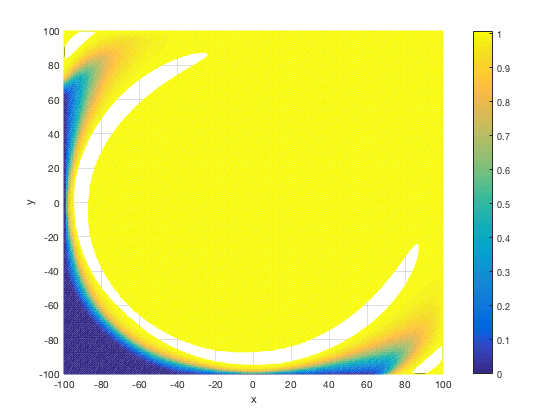}
\centerline{(d): mesh135, $N_e=51,200$}
\end{minipage}
\\
\vspace{5mm}
\begin{minipage}[b]{2in}
\includegraphics[width=2in]{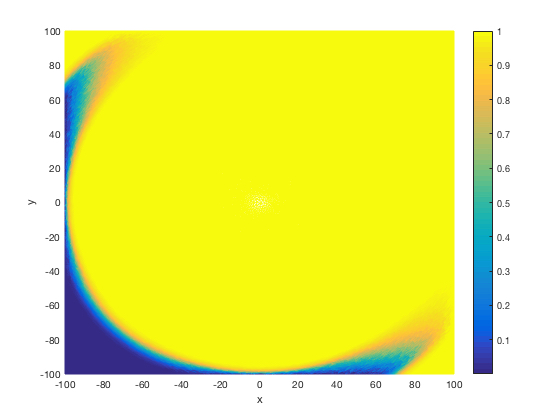}
\centerline{(e): meshDMP, $N_e=51,146$}
\end{minipage}
\end{center}
\caption{Example \ref{ex3}. Filled contour plot of the numerical solutions at $T = 40$ obtained from different meshes with $\Delta t = 0.1$.
The regions of  undershoot and overshoot are displayed as empty white areas.}
\label{ex3-soln2}
\end{figure} 

\begin{table}[hbt]
\caption{Numerical solutions obtained from different meshes and methods for Example \ref{ex3} with lumped mass and reaction terms.}
\vspace{2pt}
\centering
\begin{tabular}{cc|cccc}
\hline \hline
Mesh & solution & EM & IM & HEIM I & HEIM II  \\
\hline \hline
 meshIso & $u_{min}$ & -1.2e-3 & -1.2e-3 & -1.2e-3 & -1.2e-3 \\
 $N_e=51,910$ & $u_{max}$ & 1.0111 & 1.0111 & 1.0111 & 1.0111 \\
\hline
 mesh45 & $u_{min}$ & -1.5e-2 & -1.5e-2 & -1.5e-2 & -1.5e-2 \\
 $N_e=51,200$ & $u_{max}$ & 1.0087 & 1.0087 & 1.0087 & 1.0087 \\
\hline 
 mesh135 & $u_{min}$ & 0 & 0 & 0 & 0 \\
 $N_e=51,200$ & $u_{max}$ & 1.0068 & 1.0068 & 1.0068 & 1.0068 \\
\hline
 meshAcute & $u_{min}$ & -1.0e-2 & -1.0e-2 & -1.0e-2 & -1.1e-2 \\
 $N_e=51,200$ & $u_{max}$ & 1.0115 & 1.0115 & 1.0115 & 1.0115 \\
\hline
 meshDMP & $u_{min}$ & 0 & 0 & 0 & 0 \\
 $N=51,146$ & $u_{max}$ & 1 & 1 & 1 & 1 \\
\hline \hline
\end{tabular} \\
\label{ex3-lump}
\end{table}  

\end{exam}

\section{Conclusions}
\label{sec-con}

Nagumo-type equations arise from different fields and have been widely studied. However, preservation of the nonnegativity and boundedness in the numerical solution remains an important and challenge task.
Some results have been obtained for the finite difference solution with isotropic diffusion, but little is
known for the finite element solution especially with anisotropic diffusion. 

In the previous sections we have studied the numerical solution of Nagumo-type equations
with both isotropic and anisotropic diffusion. Linear finite elements on simplicial meshes and
the backward Euler scheme have been used for the spatial and temporal discretization, respectively.
Four different treatments for the nonlinear reaction term have been considered, including the explicit method (EM),
the fully implicit method (IM), and two hybrid explicit-implicit methods (HEIM I and HEIM II). 
The conditions for the mesh and the time step  have been developed for the numerical solution
to preserve the nonnegativity and/or boundedness of the solution of the continuous problem.
They are stated in Theorems \ref{thm-em}, \ref{thm-im}, \ref{thm-heim1}, and \ref{thm-heim2}. 
Roughly speaking, the mesh conditions require that the mesh be at least acute in the metric $\D^{-1}$
where $\D$ is the diffusion matrix while the time step conditions
ask the time step to be bounded below and above. An only exception is HEIM II which does not
have an upper bound placed on the time step.
Moreover, if the mesh is uniformly acute in the metric $\D^{-1}$ as it is being refined,
the time step conditions essentially become
\[
\mathcal{O}(h^2) \le \Delta t_n \le \mathcal{O}(1),
\]
which can be satisfied easily when the mesh is sufficiently fine.

We have also studied the effects of lumping on the mass matrix and reaction term. The analysis shows that
lumping not only removes the low bound requirement for the time step but also relaxes the requirement
on the mesh for the numerical solution to preserve the nonnegativity and/or boundedness. Indeed, with lumping
the conditions only require that the mesh be nonobtuse in the metric $\D^{-1}$
and the time step be bounded as
\[
\Delta t_n \le \mathcal{O}(1) .
\]
Once again, there is no upper bound for $\Delta t_n$ for HEIM II. 

HEIM II leads to the weakest conditions among the four approximations
for the reaction term. However, numerical experiment shows that other approximations give
smaller temporal discretization errors. Numerical examples also confirm the first order
convergence in time and second order convergence in space for the scheme with all four approximations to
the reaction term.

\vspace{20pt}

\textbf{Acknowledgments.} The authors are grateful to Erik Van Vleck for his valuable comments
in improving the quality of the paper.

\end{document}